%% file: normeq.tex
\documentclass[reqno]{amsart}
\usepackage{amsmath,amsfonts,amsthm}
\usepackage[margin=1in]{geometry}
\usepackage[inline]{enumitem}
\usepackage{graphicx} 
\usepackage{tikz}
\usepackage{graphicx}
\usepackage{subcaption}
\usetikzlibrary{calc}
\usepackage[authoryear,round,sort]{natbib}

\newtheorem{theorem}{Theorem}[section]

\newtheorem{lemma}[theorem]{Lemma}
\theoremstyle{definition}
\newtheorem{definition}[theorem]{Definition}

\theoremstyle{remark}

\newtheorem{remark}[theorem]{Remark}

\newcommand{\norm}[1]{\lVert#1\rVert}
\def\R{\mathbb{R}}
\def\B{\mathfrak{B}}
\def\H{\mathfrak{H}}
\def\Z{\mathfrak{Z}}
\DeclareMathOperator{\tr}{tr}
\DeclareMathOperator{\St}{St}
\DeclareMathOperator{\im}{im}
\newcommand{\dualcomplex}{\raisebox{0.2ex}{\(\star\)}}
\newcommand{\whitney}[2]{\mathcal{P}_1^-\Lambda^{#1}(#2)}

\numberwithin{equation}{section}

\title[DEC Convergence and Stability in Two Dimensions]%
      {Convergence and Stability of Discrete Exterior Calculus for the Hodge Laplace Problem in Two Dimensions}

\author[C.~Zhu, S.~H.~Christiansen, K.~Hu, A.~N.~Hirani]{%
Chengbin Zhu$^{1}$, Snorre H.~Christiansen$^{2}$, Kaibo Hu$^{3}$, Anil N.~Hirani$^{1}$%
}

\thanks{$^{1}$Department of Mathematics, University of Illinois at Urbana-Champaign, Urbana, U.S.A.; 
\texttt{cz43@illinois.edu}, \texttt{hirani@illinois.edu}.}
\thanks{$^{2}$Department of Mathematics, University of Oslo, Oslo, Norway; 
\texttt{snorrec@math.uio.no}.}
\thanks{$^{3}$Mathematical Institute, University of Oxford, Oxford, United Kingdom; 
\texttt{kaibo.hu@maths.ox.ac.uk}.}

\date{First posted on arXiv on May 13, 2025}

\begin{document}

\begin{abstract}
We prove convergence and stability of the discrete exterior calculus (DEC) solutions for the Hodge-Laplace problems in two dimensions for families of meshes that are non-degenerate Delaunay and shape regular. We do this by relating the DEC solutions to the lowest order finite element exterior calculus (FEEC) solutions. A Poincar\'e inequality and a discrete inf-sup condition for DEC are part of this proof. We also prove that under appropriate geometric conditions on the mesh the DEC and FEEC norms are equivalent. Only one side of the norm equivalence is needed for proving stability and convergence and this allows us to relax the conditions on the meshes. 
\end{abstract}

\maketitle

\section{Introduction}

Discrete exterior calculus (DEC) is a framework for discretizing exterior calculus on simplicial complexes that are intended to approximate manifolds in arbitrary dimension~\citep{Hirani2003}. The method is independent of any considerations of charts and transition functions required for manifold definition and also does not depend on global embedding in $\R^N$. The main application has been to numerical methods for partial differential equations, both linear and nonlinear~\citep{MoHiSa2016,WaJaHiSa2023,MoRaSaRo2025}. DEC operators act on cochains which would be referred to as degrees of freedom in finite element (FE) language. The DEC cochains are identical to the degrees of freedom of the lowest order finite element exterior calculus (FEEC)~\citep{ArFaWi2010, Arnold2018}, and the differential operators also directly correspond. Actually these constructs date back to work by de Rham and Whitney. From a FE point of view DEC provides modified mass matrices that are computed using measures of the primal and dual mesh entities. This modification can be  analyzed as a variational crime using Strang's lemma.

The question of convergence in DEC can be asked by interpolating the cochains on the mesh to get piecewise polynomials. There is numerical evidence for convergence of DEC methods for the Hodge-Laplace problem in two dimensions~\citep{Kalyanaraman2015}, but a general proof of convergence and stability remained open.\footnote{We note that independently and simultaneously with this work~\cite{GuPo2025} proved the convergence and stability of the Hodge-Laplace problem on well-centered triangulations, employing generalized Whitney forms on the dual mesh.}
In dimension $n$ the Hodge-Laplace problem consists of two families of $n+1$ problems. The problems are indexed by the degree $k=0,\ldots,n$ of the forms involved. One family is for natural boundary conditions and one for essential. In vector calculus language $k=0$ case with natural boundary conditions corresponds to the scalar Poisson equation with homogeneous Neumann boundary condition and the $k=n$ case to the scalar problem with Dirichlet boundary condition. For the essential boundary condition family, $k=0$ corresponds to Dirichlet and $k=n$ to Neumann boundary conditions. In two dimensions the $k=1$ case uses the vector Laplacian $\operatorname{curl}\,\operatorname{curl} - \operatorname{grad}\,\operatorname{div}$. See~\cite{ArFaWi2010,Arnold2018} for introduction and analysis of these problem in FEEC. The analysis of the scalar case can also be found in~\cite{Cia91} and the vector case is an example of mixed finite element methods~\citep{RobTho91,BofBreFor13}. In the case of DEC, the convergence and stability for $k=0$ case with essential Dirichlet boundary conditions was proved in~\cite{ScTs2020} (although there is a gap in their proof as we point out later). Their proof is for shape regular families of well-centered triangulations~\citep{VaHiGuRa2010} (which means acute triangulations in two dimensions) in general dimension. In two dimensions the $k=0$ case is implied by various finite element results as we summarize in~\S\ref{subsec:implied}. The $k=1$ and $k=2$ cases in two dimension have been open.

\subsection{Summary of results} The results we prove are summarized here. See later sections for the notation and terminology. 

For piecewise linear surface $\Omega$ we prove the following results:
\begin{enumerate*}[label=\textit{(\roman*)}]
    \item In \S\ref{sec:norm-acute} Theorem~\ref{thm:acute-equivalence} shows that the DEC and FEEC norms are uniformly equivalent for uniformly acute families of triangulations of $\Omega$. The acute case is easiest to explain. This restriction on triangulation is relaxed in a later section. We point out the gap in the norm equivalence proof in~\cite{ScTs2020} in Remark~\ref{rem:ScTs-error}.
    \item The results of~\S\ref{sec:norm-acute} are extended in~\S\ref{sec:norm-delaunay} to allow for right-angled and obtuse triangles. Theorem~\ref{thm:delaunay-equivalence} proves that the DEC and FEEC norms are equivalent for uniformly Delaunay uniformly boundary acute families of triangulations which are curvature bounded. (This last condition is only needed for 0-forms.)  Note that only one side of norm equivalence is needed for stability and convergence proofs which are proved in~\S\ref{sec:convergence}.
    \item The difference between DEC and FEEC inner products in terms of the mesh parameter $h$ is estimated in terms of the $H\Lambda^k(\Omega)$ norm in~\S\ref{sec:ip}. Theorem~\ref{thm:ip-error} proves this estimate for DEC regular families, that is, triangulation families that are shape regular and non-degenerate Delaunay.
    \item For DEC regular triangulations two versions of a DEC Poincar\'e equality are proved in Theorem~\ref{thm:PI}. One version is in terms of FEEC norms and the other in terms of DEC norms, both using orthogonality with respect to the DEC inner product. 
\end{enumerate*}

For domains $\Omega \subset \R^2$ we prove the following results relating to DEC convergence and stability in~\S\ref{sec:convergence}. (The planarity assumption is needed because we use FEEC machinery from~\cite{ArFaWi2010}.) All of these are for a shape regular family of non-degenerate Delaunay triangulations, that is, for DEC regular triangulations. 
\begin{enumerate*}[label=\textit{(\roman*)},resume]
    \item The DEC Poincar\'e inequality in Theorem~\ref{thm:PI} leads to a discrete inf-sup condition for DEC that follows from the stability result proved in Theorem~\ref{thm:inf-sup-FEEC-norm}. This is stated in terms of FEEC norm using FEEC harmonic forms.
    \item The stability result is used to estimate the difference between DEC and FEEC solutions for the weak mixed formulation of the Hodge-Laplace problem with natural boundary conditions for 0, 1 and 2 forms. This is the content of Theorem~\ref{thm:error-FEEC-H} which is formulated in terms of FEEC norm and FEEC harmonic forms.
    \item Theorem~\ref{thm:harmonic-comparison} proves a comparison result between DEC and FEEC harmonic forms. 
    \item This allows for a stability result in Theorem~\ref{thm:inf-sup-DEC-norm} formulated in terms of DEC norm and DEC harmonic forms.
    \item An error estimate between DEC and FEEC solutions for the Hodge-Laplace problem are then proved using DEC harmonic forms in Theorem~\ref{thm:error-DEC-H}.
\end{enumerate*}

\section{Background and setup} \label{subsec:background}

\subsection{Triangulations and DEC spaces}
Let $X$ be a simplicial complex of dimension 2 which is realizable as a $C_0$ topological manifold. $X$ may be embedded in $\R^N$ for some $N \ge 2$ or may be specified by gluing information on individual triangles, each of which is separately realized as a Euclidean triangle. An example of $X$ is a piecewise linear surface mesh embedded in $\R^3$ or a triangulation of a planar domain in $\R^2$ with piecewise linear boundary. Let $\Omega$ be the underlying space of a fixed geometric realization of $X$. To fix ideas it is convenient to let $X$ be a fixed triangulation of a piecewise linear triangle mesh surface $\Omega$ embedded in $\R^3$ or triangulation of a planar domain. We will be considering a family $\{X_h\}$ of meshes parametrized by a real number $h > 0$ denoting also mesh-size. Each member of $\{X_h\}$ triangulates the same underlying space $\Omega$.

\begin{remark}
    When using infinite dimensional Sobolev spaces defined on $\Omega$ and when referring to convergence to solutions in these spaces, we will require $\Omega \subset \R^2$ since the stability and convergence theory in FEEC is most well developed for this case. When proving convergence of DEC to FEEC or showing norm equivalence, the planarity of $\Omega$ is not assumed and it can be a piecewise linear surface.
\end{remark}

The following geometric conditions on triangulations and/or families of triangulations of $\Omega$ will be relevant:
\begin{enumerate}[label=(\roman*)]
\item Acute: $\theta <\pi/2$ for all angles $\theta$.
\item Uniformly acute: $\theta \le \pi/2 - \delta_{\pi/2}$ for some fixed $\delta_{\pi/2} > 0$ for all angles $\theta$.
\item Boundary acute: $\theta_\partial < \pi/2$ for all angles $\theta_\partial$ opposite to boundary edges.
\item Uniformly boundary acute: $\theta_\partial \le \pi/2 - \delta_{\pi/2}$ for some fixed $\delta_{\pi/2} > 0$ for all angles $\theta_\partial$ opposite to boundary edges.
\item Delaunay: $\theta_i + \theta_j \le \pi$ for all pairs of angles $\theta_i, \theta_j$ opposite to internal edges. Equivalently every circumcircle has no mesh vertex in its interior.
\item Non-degenerate Delaunay: $\theta_i + \theta_j < \pi$ for all pairs of angles $\theta_i, \theta_j$ opposite to internal edges. Equivalently every circumcircle has exactly 3 vertices on it.
\item Uniformly Delaunay: $\theta_i + \theta_j \le \pi -\delta_\pi$ for some fixed $0 < \delta_\pi < \pi$ for all pairs of angles $\theta_i, \theta_j$ opposite to internal edges.
\item Shape regular: $\theta \ge \delta_0$ for some fixed $\delta_0 > 0$ for all angles $\theta$. Equivalent conditions are that the ratio $r/R$ of inscribed circle radius to circumcircle radius is bounded below and that the ratio of triangle areas to circumradii is bounded below. See Appendix.
\item DEC regular: Shape regular and non-degenerate Delaunay.
\item Curvature bounded: Angle defect at each vertex $v$ is bounded below by $-N$, for some constant $N > 0$. Angle defect is $2\pi - \sum_{i} \theta_i$ where the sum is over all triangles $i$ containing $v$ and $\theta_i$ is the angle at $v$ of triangle $i$.
\end{enumerate} 
Shape regularity is interesting only for a family of triangulations, since any given mesh is shape regular. Every acute triangulation is non-degenerate Delaunay and uniformly acute is uniformly Delaunay. The name curvature bounded is reasonable since angle defect is related to a commonly used definition for discrete Gaussian curvature of a piecewise linear triangle mesh surface.

In two dimensions the relevant spaces for DEC are $C^k(X)$, the primal $k$-cochains and the corresponding dual cochains $D^{2-k}(\dualcomplex X)$ on the circumcentric dual cell complex $\dualcomplex X$ for $k=0,1,2$~\citep{Hirani2003}. The dual of a vertex $\sigma^0$ is the union of all triangles of the form $[c(\sigma^0)=\sigma^0, c(\sigma^1), c(\sigma^2)]$ where $c(\sigma^1), c(\sigma^2)$ are \emph{circumcenters} of all edges and triangles satisfying $\sigma^0 \prec \sigma^1 \prec \sigma^2$. In a planar domain this is the Voronoi cell of the vertex. Dual of an edge is the union of edges $[c(\sigma^1), c(\sigma^2)]$ and dual of a triangle is its circumcenter. The volumes of the dual cells are computed using signed volumes of the component simplices. For non-degenerate Delaunay triangulations that are boundary acute the resulting volumes are positive~\cite{HiKaVa2013}.

\subsection{Hodge-Laplace problem}
 We will use the notation of~\cite{ArFaWi2010} for the spaces relevant to the Hodge-Laplace problem in the FEEC context. The domain spaces for the exterior derivative are
\[
    H\Lambda^k(\Omega) := \{\omega \in L^2\Lambda^k(\Omega) \;|\; d\omega \in L^2\Lambda^{k+1}(\Omega)\} \quad\text{and}\quad 
    \mathring{H}\Lambda^k(\Omega) := \{\omega \in H\Lambda^k(\Omega) \;|\; \tr_{\partial\Omega}\omega = 0\}\, ,
\]
with inner product 
\begin{equation}\label{eq:Hlambda-ip}
    \langle \omega, \eta\rangle_{H\Lambda^k} := \langle\omega,\eta\rangle_{L^2\Lambda^k} + \langle d\omega,d\eta\rangle_{L^2\Lambda^{k+1}}\,
\end{equation}
and an analogous one on $\mathring{H}\Lambda^k$.
Let $\B^k = \im d^{k-1}$ and $\Z^k = \ker d^k$. Then $\H^k = \Z^k \cap \B^{k\perp}$ is the space of harmonic forms with the orthogonality with respect to the inner product~\eqref{eq:Hlambda-ip}. The spaces $\mathring{\B}^k$, $\mathring{\Z}^k$ and $\mathring{H}^k$ are defined similarly.

Similarly, we define
$$
    H^{\ast}\Lambda^k(\Omega) := \{\omega \in L^2\Lambda^k(\Omega) \;|\; \delta \omega \in L^2\Lambda^{k-1}(\Omega)\} \quad\text{and}\quad 
    \mathring{H}^{\ast}\Lambda^k(\Omega) := \{\omega \in H^{\ast}\Lambda^k(\Omega) \;|\; \tr_{\partial\Omega}\star\omega = 0\}\, .
$$

Given $f \in L^2\Lambda^k(\Omega)$ the Hodge-Laplace problem with natural boundary conditions is to find $u \in H\Lambda^k(\Omega)\cap \mathring{H}^{\ast}\Lambda^{k}(\Omega)$, $u \perp \H^k$,  such that
\begin{equation}\label{eq:Hodge-Laplace-natural}
\langle du, dv\rangle+\langle \delta u, \delta v\rangle=\langle f-p, v\rangle, \quad \forall v\in H\Lambda^k(\Omega)\cap \mathring{H}^{\ast}\Lambda^{k}(\Omega)
\end{equation}
where $p = P_{\H^k} f$ is the $L^2$ orthogonal projection of $f$ into $\H^k$. Here $\Delta$ is the Hodge-Laplacian $d\delta + \delta d$ where $d$ is the exterior derivative and $\delta$ the codifferential. ($d\delta + \delta d$ reduces to $\delta d$ for $k=0$ and $d \delta$ for $k=2$.) The Hodge-Laplace problem with essential boundary conditions is to find $u \in \mathring{H}\Lambda^k(\Omega)\cap H^{\ast}\Lambda^{k}(\Omega)$, $u \perp \mathring{\H}^k$,
\begin{equation}\label{eq:Hodge-Laplace-essential}
\langle du, dv\rangle+\langle \delta u, \delta v\rangle=\langle f-p, v\rangle, \quad \forall v\in \mathring{H}\Lambda^k(\Omega)\cap {H}^{\ast}\Lambda^{k}(\Omega).
\end{equation}
 
The FEEC spaces relevant to this work are $\whitney{k}{\Omega}$, the lowest order piecewise polynomial Whitney $k$-forms~\citep{Whitney1957,ArFaWi2010,Arnold2018} for $k=0,1,2$. These are finite dimensional subspaces of $H\Lambda^k(\Omega)$ and of $\mathring{H}\Lambda^k(\Omega)$ after incorporating boundary conditions. For $k=0$ this is the space of continuous linear polynomials and for $k=2$ piecewise constants. For $k=1$ the vector field proxies are N\'ed\'elec or (rotated) Raviart-Thomas elements. 

\subsection{DEC and FEEC inner products and norms}

Let $W: C^k(X)\to \whitney{k}{\Omega}$ be the Whitney map that interpolates the cochains and $R:\whitney{k}{\Omega} \to C^k(X)$ the de Rham map consisting of integration of the Whitney forms on the $k$-simplices.
For $\omega,\eta \in \whitney{k}{\Omega}$ the \emph{DEC inner product} $\langle \omega,\eta\rangle_D$ is
\begin{equation}\label{eq:dec-ip}
    \langle\omega,\eta\rangle_D := \sum_{\sigma \in X}
    \frac{|\star \sigma|}{|\sigma|} \, R(\omega)(\sigma)\;R(\eta)(\sigma)\, ,
\end{equation}
where the sum is taken over all $k$-simplices $\sigma \in X$ and $R(\omega)(\sigma)$ is the value of the cochain $R(\omega) \in C^k(X)$ on $k$-simplex $\sigma$. Here $|\sigma|$ denotes the $k$-volume of $\sigma$ which is 1 for a vertex, length for an edge and area for a triangle and \(|\star \sigma|\) is the $(2-k)$-volume of the dual cell $\star \sigma$. By analogy we will call the $L^2\Lambda^k(\Omega)$ inner product restricted to $\whitney{k}{\Omega}$ the \emph{FEEC inner product} and denote it $\langle \omega, \eta\rangle_F$. This is defined 
\begin{equation}\label{eq:feec-ip}
    \langle \omega,\eta\rangle_F := \int_{\Omega} \langle \omega, \eta \rangle \, d\mu\, ,
\end{equation}
where the inner product in the integrand is taken pointwise, and \(d\mu\) is the volume form on the domain $\Omega$. The corresponding norms will be denoted $\norm{\omega}_D$, $\norm{\omega}_F$ etc.

\begin{remark}
    We could have defined these inner products and norms for cochains $\omega,\eta\in C^k(X)$ instead of for Whitney forms $\whitney{k}{\Omega}$. In that case the DEC inner product would use $\omega$ and $\eta$ instead of $R(\omega)$ and $R(\eta)$ and FEEC inner product would use $W(\omega)$ and $W(\eta)$ instead of $\omega$ and $\eta$. Also note that the FEEC versions could have been written without the $F$ subscript since these are the standard $L^2$ inner products and norms acting on the fields, but we retain the $F$ to distinguish between DEC and FEEC.
\end{remark}

As is usual in FE, these inner products and norms can be implemented using mass matrices. If a mesh has $N_k$ number of $k$-simplices, $k=0,1,2$, the DEC Hodge star $*^D_k$ (mass matrix for $k$-cochains) is an $N_k\times N_k$ diagonal matrix. The diagonal entry corresponding to a $k$-simplex $\sigma$ is $|\star\sigma|/|\sigma|$ where $|\star \sigma|$ is the measure of the \emph{circumcentric} dual cell $\star\sigma$ and $|\sigma|$ that of $\sigma$. As mentioned above these diagonal Hodge star matrices in DEC have positive entries if the mesh is non-degenerate Delaunay and boundary acute. Thus $\star^D_k$ can be used to define a norm and inner product. The FEEC mass matrices $*^F_k$ are constructed by using Whitney form basis corresponding to the $k$-simplices of $X$ in~\eqref{eq:feec-ip}. The FEEC mass matrices for 0-forms and 1-forms are sparse but not diagonal. For $\omega,\eta\in \whitney{k}{\Omega}$ if $w, v\in \R^{N_k}$ are the vectors of values of $R(\omega), R(\eta)\in C^k(X)$ then $\langle\omega,\eta\rangle_D = w^T *_k^D v$ and $\langle\omega,\eta\rangle_F = w^T *_k^F v$.

\subsection{DEC results implied by finite element results}
\label{subsec:implied}
We first point out that in fact the convergence and stability for $k=0$ case for the DEC solution of~\eqref{eq:Hodge-Laplace-natural} for families of DEC-regular triangulations of $\Omega \subset \R^2$ follows from a mass lumping result in finite elements~\citep{HaLa1993}. (See also \cite{ChHa2011}.) This is because DEC Hodge star $*^D_1$ is the same as the mass lumped mass matrix in~\cite{HaLa1993,BoKe1999,ChHa2011} as shown next.

For a triangle $T$ with vertices $v_0, v_1, v_2$, let \( e_2 = [v_0, v_1] \) denote the edge opposite vertex \( v_2 \), and let \( \star e_2 \) be its dual edge. Let $d\mu$ be the volume form on $T$. Then, the DEC Hodge star $*^D_1$ entry for edge $e_2$ satisfies:
\begin{equation}\label{eq:lump}
  -\int_T\langle d\lambda_0, d\lambda_1 \rangle d\mu= \frac{|\star e_2|}{|e_2|},  
\end{equation}
where \( \lambda_i \) denotes the barycentric coordinate associated with vertex \( v_i \), and \( |e_2| \), \( |\star e_2| \) are the lengths of the edge and its dual, respectively. Let the angles at $v_0, v_1, v_2$ be $A,B,C$ and the side lengths opposite these $a,b,c$ and $\mu$ the area of $T$. Then $|d\lambda_0| = 1/(c\sin B), |d\lambda_1| = 1/(c\sin A)$ implies that
\[
    -\int_T\langle d\lambda_0, d\lambda_1 \rangle d\mu = 
    \frac{\cos C}{c^2 \sin A \sin B} \mu = 
    \frac{ab \sin C \cos C}{2c^2 \sin A \sin B} = 
    \frac{\sin A \sin B \sin C \cos C}{2\sin^2 C \sin A \sin B} = 
    \frac{\cos C}{2\sin C} = \frac{|\star e_2|}{|e_2|}\, .
\]
Let $M_1$ be the lumped mass matrix with diagonal entries $-(\int_{T_-} \langle d\lambda_i, d\lambda_j\rangle d\mu + \int_{T+}\langle d\lambda_i, d\lambda_j\rangle d\mu)$ for internal edge $[i,j]$ shared by triangles $T_-$ and $T_+$ and $-\int_T \langle d\lambda_i, d\lambda_j\rangle d\mu$ for boundary edge in triangle $T$. For a shape regular family  of triangulations of $\Omega$ the finite element solution to~\eqref{eq:Hodge-Laplace-natural} for $k=0$ is stable and convergent. This is because the norm approximated using the lumped matrix $M_1$ is bounded by the finite element $L^2(\Omega)$ norm~\citep{HaLa1993}. 

Since $M_1 = *_1^D$, for the non-degenerate Delaunay case the stiffness matrices for the scalar problem are identical in DEC and FE and the right hand sides are different, so the methods are nearly identical in that case.

\begin{remark}\label{rem:star1}
The observation~\eqref{eq:lump} also implies a unique characterization of $*_1^D$ in two dimensions. Let $K$ be the stiffness matrix for the scalar problem using the FE mass matrix without lumping. For non-degenerate Delaunay meshes the lumped mass matrix $M_1$ (and thus $*_1^D$) is the \emph{unique} diagonal matrix which yields stiffness matrix $K$~\citep{HaLa1993,BoKe1999}.
\end{remark}

\subsection{Previous work}\label{sec:previous}
As mentioned earlier the only convergence and stability result known for the DEC setup before this work and~\cite{GuPo2025}, whether implied by mass lumping or other FE results, or proved directly for DEC, was for the $k=0$ case, with the only direct DEC result being for the essential Dirichlet boundary condition for $k=0$ in general dimension in~\cite{ScTs2020}. They prove this using a Poincar\'e inequality for DEC which follows from their Theorem 4.1 on norm equivalence. The next remark is about the gap in that theorem.
\begin{remark}\label{rem:ScTs-error}
In~\cite{ScTs2020} it is claimed the DEC and FEEC norms are equivalent for a shape regular family of acute triangulations~\cite[Theorem 4.1]{ScTs2020}. For a counterexample consider a shape regular family of meshes in which there is a pair of triangles with a shared edge $e$ such that the sum of opposite angles approaches $\pi$. Now consider a 1-cochain $\alpha$ taking the value 1 on $e$ and 0 on others. Then $\norm{\alpha}_D$ approaches 0 whereas $\norm{\alpha}_F$ remains bounded above 0.
\end{remark}

Mass-lumping has been used to relate finite element methods to other numerical methods, in particular Finite Difference Time Domain methods (such as the Yee scheme)~\citep{Jol03} and Finite Volume methods~\citep{BarMaiOud96} 
In~\cite{ChHa2011} it was used to construct gauge invariant methods for the Schr\"odinger equation that can be compared with a standard FE method. Here we establish similar results in two dimensions, relating DEC and FEEC. Although the mass lumping on Whitney 1-forms coincides exactly with the DEC mass matrix, the mass lumping of 0-forms used in~\cite{ChHa2011} is different from the DEC Hodge star 0 matrix. 

In addition to mass lumping, there are other ways to related DEC to various frameworks for discretizing PDEs. For example, the linear system discretizing the mixed formulation of the Hodge-Laplace problem in lowest order FEEC can be turned into a DEC method by replacing mass matrices. The use of primal and dual meshes in DEC is related to similar ideas in staggered grid methods~\citep{Yee1966,HaWe1965,RiThKlSk2010} and the use of cochains and coboundary operator as exterior derivative is related to ideas in~\cite{Bossavit1988,Mattiussi1997,SeSeSeAd2000}. None of the above however give a proof of convergence for the Hodge-Laplace problem for DEC for all $k$. For example, one primal-dual method for simplicial meshes is the covolume method of~\cite{Nicolaides1992}. The error analysis in~\cite{Nicolaides1992} is for the scalar case (or rather $k=2$ case) as application of the analysis of the div-curl problem. The analysis appears to be difficult to generalize to other degrees forms because it relies on a particular orthogonal decomposition that may not generalize. Subsequent work~\citep{NiWa1998} extended the covolume approach to Maxwell’s equations in three dimensions. The $k=2$ case for the covolume method using a mixed formulation has also been analyzed~\citep{ChKwVa1998}.

In~\cite{AdCaHuZi2021} the inf-sup condition is analyzed in the case of Maxwell's equation. They assume $\operatorname{div} B = 0$ as a condition for inf-sup analysis which is relevant in Maxwell's but not directly applicable (although some of the techniques could be used in the Hodge-Laplace problem). Both~\cite{AdCaHuZi2021} and~\cite{ScTs2020} use norm equivalence in their paper, but do not mention the fact that the norm equivalence constant depends on the mesh geometry and so a specific mesh shape condition must be imposed in addition to shape regularity. We point out which geometric conditions are appropriate for norm equivalence to hold. 


\section{Norm equivalence for uniformly acute triangulations}
\label{sec:norm-acute}

We first prove the equivalence of DEC and FEEC norm for acute families of meshes triangulating a piecewise linear surface mesh $\Omega$. For acute triangulations we show that uniform acuteness is sufficient for the existence of mesh-independent constant ensuring the equivalence of DEC norm and FEEC norm on 1-forms. In the next section the acuteness condition is relaxed.


Let $T$ be a triangle with edge vectors $e_i, e_j, e_k$ where the edges are oriented arbitrarily. Define the number $\delta_T$ for $T$ be
\[
    \delta_T := \min_{i \ne j \ne k} 
    \frac{\vert{\langle e_i,e_j\rangle}\vert}{\norm{e_k}^2}\, .
\]
Another description of $\delta_T$ is given by angles. If the angles of $T$ are $A, B, C$, then 
\[
    \delta_T = \min \left\{\frac{\sin A \sin B \cos C}{\sin^2 C}, \frac{\sin B \sin C \cos A}{\sin^2 A}, \frac{\sin C \sin A \cos B}{\sin^2 B} \right\}\, .
\]
\begin{remark}\label{rem:sa-lambda}

A family $\{X_h\}$ of triangulations of $\Omega$ being uniformly acute is equivalent to existence of a global constant $\delta > 0$ independent of $h$ such that $\delta_T \ge \delta$ for all triangles $T$ in the family. This is because first uniformly acute implies lower bound for angles. $A = \pi - B - C \ge \pi - 2(\pi/2-\delta_{\pi/2}) \ge 2\delta_{\pi/2}$, therefore, uniformly acute gives use a lower bound of $\delta_T$ since $\frac{\sin A\sin B\cos C}{\sin^2 C} \ge \sin^2 2\delta_{\pi/2} \cos \pi_{\delta/2}$. On the other direction, assume $C$ is the largest angles in a triangle, then we claim $C \le \pi/2 - \arcsin \delta$. This is because $\sin A \sin B \le \sin^2 C$, then $\cos C \ge \delta$.
\end{remark}

\subsection{Norm of Whitney 1-forms for a single triangle}


\begin{lemma}[Norm inequalities for basis forms] Let $T$ be a triangle and $\omega_i \in \whitney{k}{T}$ the basis Whitney form corresponding to edge $i$. Then there exist constants $c_1(\delta_T), c_2(\delta_T)$ such that
    \[
    c_1(\delta_T) \norm{\omega_i}_F^2 \le \norm{\omega_i}^2_D  \le  c_2(\delta_T) \norm{\omega_i}_F^2\, .
    \]
\end{lemma}\label{lem:basiseq}
\begin{proof}
    It suffices to prove the result for a single basis Whitney form. Let the vertices of $T$ be $v_0, v_1, v_2$. Without loss of generality, we choose $\omega_2$ corresponding to the edge opposite to $v_2$. An easy computation shows that the FEEC-norm of $\omega_2$ given by:
    \[
    \norm{\omega_2}^2_F = \frac{\mu}{6} (h_{00} - h_{01} + h_{11}),
    \]  
where \(\mu\) is the area of $T$, and \( h_{ij} = \langle d\lambda_i, d\lambda_j \rangle \), with \(\lambda_i\) denoting the barycentric coordinates corresponding to vertex $v_i$. Thus, we obtain:
\begin{equation*}
    \begin{aligned}
        \norm{\omega_2}^2_F &= \frac{\mu}{6} (h_{00} - h_{01} + h_{11}) = \frac{\mu}{6} \left( \frac{1}{c^2 \sin^2 B} + \frac{\cos C}{c^2 \sin A \sin B} + \frac{1}{c^2 \sin^2 A} \right) \\
        &= \frac{abc}{12R} \cdot \frac{4R^2(a^2 + ab\cos C + b^2)}{a^2 b^2 c^2} = \frac{c^2 + 3ab\cos C}{12 ab \sin C}\, .
    \end{aligned}
\end{equation*}
Following Remark~\ref{rem:star1}
$\norm{\omega_2}_D^2 = \cos C/(2\sin C)$. Thus
\[
\frac{\norm{\omega_2}_F^2}{\norm{\omega_2}_D^2} = \frac{c^2 + 3ab\cos C}{6ab\cos C} \geq \frac{1}{2}.
\]
We also have that $ab \cos C \geq \delta_T c^2$.
Using this, we derive:
\[
\frac{\norm{\omega_2}_D^2}{\norm{\omega_2}_F^2} =  \frac{6ab\cos C}{c^2 + 3ab\cos C}\geq \frac{6\delta_T}{1 + 3\delta_T}\, .
\]
Setting $c_1(\delta_T) = 6\delta_T/(1+3\delta_T)$ and $c_2(\delta_T) = 2$ the result follows.  
\end{proof}


\begin{lemma}\label{lem:F-lowerbound} Let $T$ be a triangle with $\delta_T > 0$. Then for any $\alpha \in \whitney{1}{T}$ we have that $\norm{\alpha}_D^2 \ge c_1(\delta_T)\norm{\alpha}_F^2$.
\end{lemma}
\begin{proof}
    Let $\alpha = \alpha^0 \omega_0 + \alpha^1 \omega_1 + \alpha^2 \omega_2$, then we have 
    \begin{equation*}
    \begin{aligned}
        \norm{\alpha}_F^2 &= \langle \alpha^0 \omega_0 + \alpha^1 \omega_1 + \alpha^2 \omega_2, \alpha^0 \omega_0 + \alpha^1 \omega_1 + \alpha^2 \omega_2\rangle_F
        \le 2(\norm{\alpha^0\omega_0}_F^2 + \norm{\alpha^1\omega_1}_F^2 + \norm{\alpha^2\omega_2}_F^2) \\
        &\le 2\left(\frac{1+3\delta_T}{6\delta_T}\right) (\norm{\alpha^0\omega_0}_D^2 + \norm{\alpha^1\omega_1}_D^2 + \norm{\alpha^2\omega_2}_D^2) = 2\left(\frac{1+3\delta_T}{6\delta_T}\right)\norm{\alpha}_D^2\,.
    \end{aligned}
    \end{equation*}
Thus $\frac{3\delta_T}{1+3\delta_T}\norm{\alpha}_F^2 \le \norm{\alpha}_D^2$ and setting $c_1(\delta_T) = \frac{3\delta_T}{1+3\delta_T}$, we have $\norm{\alpha}_D^2 \ge c_1(\delta_T)\norm{\alpha}_F^2$ for any $\alpha \in \whitney{1}{T}$.
\end{proof}

In order to prove the other direction for a single triangle, we need the following two lemmas.
\begin{lemma} Let $\delta_T$ be the constant for an acute triangle with angles $A,B,C$. Then there exists a constant $\gamma$ depending only on $\delta_T$ such that 
\begin{equation}\label{eq:sin-ratios}
    \frac{\sin A}{\sin B} \ge \gamma, \quad\frac{\sin B}{\sin C} \ge \gamma,\quad \frac{\sin C}{\sin A} \ge \gamma\, .
\end{equation}
\end{lemma}

\begin{proof}
    We only prove $\sin A/\sin B \ge \gamma$, the rest can be obtained similarly.
    By the definition of DEC-constant, we have $\sin A \sin C \cos B/\sin^2 B \ge \delta_T$, then using $\sin A = \sin (\pi - (B+C)) = \sin B\cos C + \cos B\sin C$, we have
    \[
    \delta_T \le \frac{\sin A \sin C \cos B}{\sin^2 B} = \frac{\sin^2 A - \sin A \sin B \cos C}{\sin^2 B}\le  \frac{\sin^2 A}{\sin^2 B}\, .
    \]
Then letting $\gamma = \sqrt{\delta_T}$, we get $\sin A/\sin B \ge \gamma$.
\end{proof}

\begin{lemma} If all angles of a triangle are $\le \pi/2 - \delta_{\pi/2}$ for some $0 < \delta_{\pi/2} < \pi/2$ there exists a constant $S > 0$, such that the shape constant for the triangle $\frac{r}{R} \ge S$.
\end{lemma}
\begin{proof}
    Using the previous lemma, we know that there exists a constant \( 1 > \gamma > 0 \) such that:
    \[
    \frac{\sin A}{\sin B} \geq \gamma, \quad \frac{\sin B}{\sin C} \geq \gamma, \quad \frac{\sin C}{\sin A} \geq \gamma\, .
    \]
    Our first goal is to establish a lower bound for the angles \( A, B, C \). Let \( \theta \) be the solution to the equation $\sin \theta = \gamma \cos \theta$.
    
    We claim that \( \theta \) serves as a lower bound for the angles. To see this, assume, for contradiction, that \( A \leq \theta \). Then, we have:
    $\sin B \leq (\sin A)/\gamma \le (\sin \theta)/\gamma = \cos \theta$. Thus
    $B \leq \pi/2 - \theta$. Since the triangle is acute, the third angle satisfies $C = \pi - A - B \geq \pi/2$.  However, this contradicts our assumption that the triangle is acute. Hence, our assumption must be false, proving that all angles \( A, B, C \) are bounded below by $ \theta = \arctan \sqrt{\delta_T}$. Next, we compute $r/R$
    \[
    \frac{r}{R} = \frac{abc}{4R^2(a+b+c)} = 
    \frac{\sin A\sin B\sin C}{\sin A + \sin B + \sin C} 
    \geq \frac{\sin^3 \theta}{3}\,.
    \]
    Letting $S = (\sin^3 \theta)/3$, gives the result.
\end{proof}

\begin{remark}
    The lower bound of the shape constant for a triangle is equivalent to having a lower bound on the interior angles of the triangle. And this lemma holds only for acute triangulations. 
\end{remark}

\begin{lemma}\label{lem:F-upperbound} For a triangle $T$ with $\delta_T > 0$  and $\alpha \in \whitney{k}{\Omega}$, there exists $c_2(\delta_T) > 0$ such that $\norm{\alpha}_D^2 \le c_2(\delta_T)\norm{\alpha}_F^2$ on $T$.
\end{lemma}
\begin{proof}
    The FEEC norm for 1-forms is scale invariant~\cite[equation (40)]{ChHa2011}. DEC-norm for 1-forms is scale invariant in two dimensions since it is a ratio of lengths. Thus we can assume that the triangle has diameter 1. Then we can consider the set $$S_{\pi/2} = \{T\, \vert\,  \text{angles of $T$ are $\le \pi$/2, and diameter($T$) = 1 }  \}\, .$$ $S_{\pi/2}$ is a compact set, therefore, there exists an $\epsilon >0$, such that $\norm{\alpha}_F \ge \epsilon \norm{\alpha}_{l^2}$. Let $\alpha = \alpha^0 \omega_0 + \alpha^1 \omega_1 + \alpha^2 \omega_2$,then 
    \begin{equation*}
        \begin{aligned}
            \norm{\alpha}^2_D & = (\alpha^0)^2 \frac{1}{2}\cot A + (\alpha^1)^2 \frac{1}{2}\cot A + (\alpha^2)^2 \frac{1}{2}\cot A \le ((\alpha^0)^2 + (\alpha^1)^2 + (\alpha^2)^2) \frac{1}{2}\cot\theta = \frac{1}{2} \cot \theta \norm{\alpha}_{l^2}^2\, ,
        \end{aligned}
    \end{equation*}
    where $\theta$ is the lower bound of angles from the last lemma. (This is shape regularity.)
    Therefore, by setting $c_2(\delta_T) = 1/(2\epsilon\sqrt{\delta_T})$, we have $\norm{\alpha}_D^2 \le c_2(\delta_T)\norm{\alpha}_F^2$ for $\alpha \in \whitney{1}{T}$.
\end{proof}

\subsection{Norm of Whitney 0-forms for a single triangle}

Equivalence between the DEC norm and the FEEC norm on Whitney 0-forms is guaranteed for any acute mesh and does not require uniform acuteness.
For 0-forms (scalar functions), i.e., the barycentric basis functions $ \lambda_i$ on a triangle $T$ with area $\mu$, $\int_T \lambda_i \lambda_j \, d\mu = \mu/12$ if $i \neq j$ and $\int_T \lambda_i \lambda_j d\mu = \mu/6$ if $i = j$. Thus the mass matrix of FEEC is
\begin{equation}\label{eq:0FEECmass}
\mu\begin{pmatrix}
1/6 & 1/12 &  1/12\\
1/12 & 1/6 & 1/12 \\
1/12 & 1/12 & 1/6
\end{pmatrix}
\end{equation}

\begin{figure}[t]
    \centering
    \input{figs/0-form-acute}
    \caption{For an acute triangle the dual area of a vertex is bounded above by triangle area and below by half of that. Here $O$ is the circumcenter of $v_0v_1v_2$. The dual of vertex $v_0$ is the quadrilateral $v_0M_1OM_2$. Its area is bounded below by the area of triangle $v_0M_1M_2$ and above by the half the area of $v_0v_1v_2$.}
    \label{fig:0-form-acute}
\end{figure}
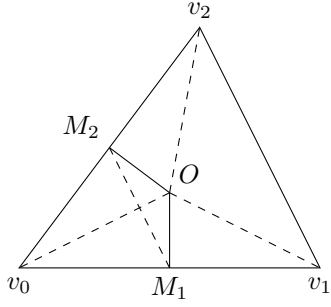

\begin{lemma}\label{lem:0-acute-T}
    For an acute triangle $T$ and $\alpha \in \whitney{0}{T}$ there exist $c_1, c_2 > 0$ such that 
    \[
    c_1\norm{\alpha}^2_F \le \norm{\alpha}^2_D  \le c_2 \norm{\alpha}^2_F\, .
    \]
\end{lemma}
\begin{proof}
See Figure~\ref{fig:0-form-acute} for this proof. Let $\mu_i$ denote the dual area of vertex $v_i$ in $T$. Then $\mu_i \ge \text{area}(OM_2M_1) = \mu/4$ and $\mu_i \le \mu/2$. Since the FEEC mass matrix~\eqref{eq:0FEECmass} is positive definite, assume it has three eigenvalues $\nu_0 \ge \nu_1 \ge \nu_2 \ge 0$, then for any $\alpha \in \whitney{0}{T}$, $\alpha = f_0\lambda_0 + f_1\lambda_1 + f_2\lambda_2$,
\[
    \norm{\alpha}_D^2 = \sum_i f_i^2\mu_i \ge \frac{\mu}{4}\sum_{i}f_i^2 = \frac{1}{4\nu_0}\mu\nu_0\sum_{i}f_i^2 \ge \frac{1}{4\nu_0}\norm{\alpha}_F^2
\]
and
\[
    \norm{\alpha}_D^2 = \sum_i f_i^2\mu_i \le \frac{\mu}{2}\sum_{i}f_i^2 = \frac{1}{2\nu_2}\mu\nu_2\sum_{i}f_i^2 \le \frac{1}{2\nu_2} \norm{\alpha}_F^2\, .
\]    
\end{proof}

\subsection{From single triangle to triangulation}

We can now prove the equivalence of DEC and FEEC norms for acute triangulations using the lemmas from the previous two subsections. 

\begin{theorem}[Norm equivalence of DEC norm and FEEC norm]\label{thm:acute-equivalence}
Given a uniformly acute family $\{X_h\}$ triangulating a piecewise linear surface $\Omega$ there exists a constant $c_1(\delta), c_2(\delta) > 0$ independent of $h$  such that
\[
   c_1(\delta)\,\norm{\alpha}^2_F \le \,\norm{\alpha}^2_D  \le\, c_2(\delta) \,\norm{\alpha}^2_F\, ,
\]
for all $\alpha \in \whitney{k}{\Omega}$ for $k=0,1,2$.
\end{theorem}    
\begin{proof}
We first address the $k=1$ case. By Lemmas~\ref{lem:F-lowerbound} and~\ref{lem:F-upperbound} for a single triangle $T$, we have 
\[
    c_1(\delta_T)\norm{\alpha|_T}^2_{F} \le \norm{\alpha|_T}^2_{D}  \le c_2(\delta_T) \norm{\alpha|_T}^2_{F}\, ,
\]
where $c_1(x) = 3x/(1+3x)$ is monotonic increasing and $c_2(x) = 1/(2\delta \sqrt{x})$ is monotonic decreasing. Thus
\begin{equation*}
          \norm{\alpha}^2_D = 
          \sum_{T\in X_h} \norm{\alpha|_T}^2_{D} \le \sum_{T\in X_h} c_2(\delta_T) \norm{\alpha|_T}^2_{F} \le \sum_{T\in X_h} c_2(\delta)\norm{\alpha|_T}^2_{F} =
          c_2(\delta)\norm{\alpha}^2_F\, ,
\end{equation*}
and
\begin{equation*}
    \norm{\alpha}^2_D = 
    \sum_{T\in X_h} \norm{\alpha|_T}^2_{D} 
    \ge \sum_{T\in X_h} c_1(\delta_T) \norm{\alpha|_T}^2_{F}
    \ge \sum_{T\in X_h} c_1(\delta)\norm{\alpha|_T}^2_{F}
    = c_1(\delta)\norm{\alpha}^2_F\, .
\end{equation*}
The $k=2$ case is trivial since then the DEC and FEEC norms are the same.
For the $k=0$ case, to go from single triangle case to the mesh family is trivial since the constants in Lemma~\ref{lem:0-acute-T} do not depend on any properties of the triangle other than the fact that it is acute.
\end{proof}


To summarize, for a family $\{X_h\}$ of acute triangulations of $\Omega$, for $\alpha \in \whitney{k}{\Omega}$ for $k=0, 2$ acuteness suffices for norm equivalence. For $k=1$ shape regularity suffices for $\norm{\alpha}_D \le c_2 \norm{\alpha}_F$ and uniform acuteness suffices for $\norm{\alpha}_D \ge c_1 \norm{\alpha}_F$. Since uniform acuteness implies shape regularity, uniform acuteness suffices for norm equivalence. Uniform acuteness is actually not necessary for norm equivalence, as we will show in the next section. 

\section{Norm equivalence for uniformly Delaunay triangulations}
\label{sec:norm-delaunay}

In this section, we establish norm equivalence for meshes that may contain obtuse triangles or right angled triangles. For such meshes, we assume that the triangulation has acute triangles at the boundary and is non-degenerate Delaunay since that ensures that the DEC norm is actually a norm by ensuring that the DEC Hodge star matrices are positive definite~\citep{HiKaVa2013}. In solving PDEs the boundary acuteness is usually not required due to the boundary conditions so non-degenerate Delaunay condition suffices.

To ensure sufficient geometric regularity, we will consider a family $\{X_h\}$ of triangulations of $\Omega$ that is a shape regular, uniformly Delaunay and uniformly boundary acute. Let $\delta_0, \delta_\pi$ and $\delta_{\pi/2}$ be the respective constants. If $\Omega \subset \R^2$ then this is all we need to show that the DEC and FEEC norms are equivalent on the family $\{X_h\}$. To allow for a non-planar $\Omega$, that is, a triangulated surface, for the 0-form norm equivalence we will impose the additional condition that we call curvature boundedness. For 1-form norm equivalence this is not needed. 

\begin{figure}[ht]
    \centering
    \input{figs/delaunay}
    \caption{In a Delaunay pair of triangles the sum of angles $C_1+C_2$ opposite to a shared edge $c$ should satisfy $C_1 + C_2\le \pi$. DEC requires nondegenerate Delaunay condition $C_1 + C_2 < \pi$ and norm equivalence requires uniform Delaunay condition $C_1 + C_2 \le \pi - \delta_\pi$ for some fixed $0 < \delta_\pi < \pi$ for all such angle pairs.}
    \label{fig:delaunay}
\end{figure}
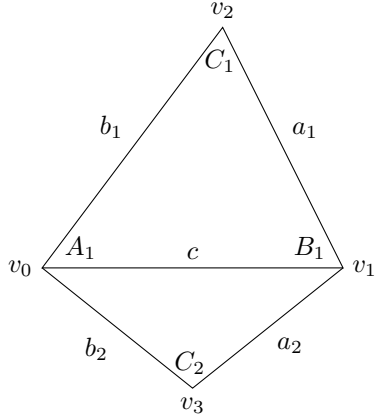

The positive definiteness of DEC Hodge star holds when pairs of triangles are considered. Since each pair is Delaunay one triangle can be obtuse as long as the sum of angles opposite the shared edge is less than $\pi$ (the non-degenerate Delaunay condition). However, we will have occasion to consider the Hodge star diagonal entries on a single triangle $T$ which may be obtuse. In that case $\norm{\,\alpha|_T\,}_D$ need not be a norm since it could be $\le 0$. In this case it will suffice to consider the absolute value for intermediate calculations and we will write $|\; \norm{\,\alpha|_T\,}_D\; |$.

\subsection{Norm of Whitney 1-forms}
We first show norm equivalence for basis elements, not in a single triangle, but for the whole mesh. 
\begin{lemma}\label{lem:basis1form-delaunay}
    Let $\Omega$ be a piecewise linear triangle mesh surface and $\{X_h\}$ a family of shape regular, uniformly Delaunay, uniformly boundary acute meshes triangulating $\Omega$. 
    Let $\omega_i^h\in \whitney{1}{X_h}$ be the Whitney form basis element associated with edge $i$ for a mesh $X_h$ in $\{X_h\}$. Then there exist $c_1, c_2 > 0$ only depending on $\delta_0, \delta_{\pi/2}$ and $\delta_\pi$ and not depending on $h$, such that 
    \[
      c_1\norm{\omega_i^h}_F^2 \le  \norm{\omega_i^h}^2_D \le c_2\norm{\omega_i^h}_F^2\, .
    \]
\end{lemma}

\begin{proof}
    If $i$ is a boundary edge, then the opposite angle is assumed uniformly acute, and the proof follows from Theorem~\ref{thm:acute-equivalence}. Now assume that $i$ is an internal edge, and hence is shared by two triangles like edge $v_0v_1$ in Figure \ref{fig:delaunay}. Then
    \begin{equation*}
        \begin{aligned}
            \norm{\omega_i^h}^2_D &= \frac{\cos C_1}{2\sin C_1} + \frac{\cos C_2}{2\sin C_2}
            = \frac{\sin C_2\cos C_1 + \sin C_1\cos C_2}{2\sin C_1 \sin C_2} = \frac{\sin(C_1 + C_2)}{2\sin C_1 \sin C_2}\, ,
        \end{aligned}
    \end{equation*}
    and
    \begin{equation*}
        \begin{aligned}
            \norm{\omega_i^h}^2_F &= \frac{c^2 + 3a_1b_1\cos C_1}{12a_1b_1\sin C_1} + \frac{c^2 +3a_2b_2\cos C_2}{12a_2b_2\sin C_2} \\
            &= c^2 \bigg(\frac{1}{12a_1b_1\sin C_1}+\frac{1}{12a_2b_2\sin C_2}\bigg)+ \frac{\sin(C_1 + C_2)}{4\sin C_1 \sin C_2}
        \end{aligned}
    \end{equation*}
Choosing $c_2 = 2$, we have $\norm{\omega_i^h}_D^2 \le c_2 \norm{\omega_i^h}^2_F$. Since $2\delta_0 \le C_1 + C_2 \le \pi - \delta_\pi$ we have 
\[
\norm{\omega_i^h}_D^2 \ge \min \{(\sin\delta_\pi)/2, (\sin 2\delta_0)/2\} > 0\, .
\]
Naming the constant $\min \{(\sin\delta_\pi)/2, (\sin 2\delta_0)/2\}$ to be $\delta_{0,\pi}$
\[
\frac{\norm{\omega_i^h}_D^2}{\norm{\omega_i^h}_F^2} \ge 
\frac{\delta_{0,\pi}}
 {2c^2((1/12a_1b_1\sin C_1)+1/12a_2b_2\sin C_2)) + \delta_{0,\pi}} \ge 
\frac{\delta_{0,\pi}}{1/(3\sin^2 \delta_0)+ \delta_{0,\pi}}\, .
\]
Choosing 
$c_1 = (3\delta_{0,\pi} \sin^2 \delta_0)/(1+ 3\delta_{0,\pi} \sin^2 \delta_0)$ 
completes the proof. Here both uniformly Delaunay and shape regularity properties are used for lower bound of DEC norm, but for upper bound, we only need shape regularity.
\end{proof}

\begin{lemma}
    Let $\Omega$, $\{X_h\}$, and $X_h$ be as in Lemma~\ref{lem:basis1form-delaunay}. Then 
    for any $\alpha_h \in \whitney{1}{X_h}$ we have that $\norm{\alpha_h}_D^2 \ge c_1\norm{\alpha_h}_F^2$.
\end{lemma}

\begin{proof}
    Let $\alpha_h = \sum_i \alpha^i_h \omega_i$, where $\omega_i$ is the Whitney form basis corresponding to edge $i$ in $X_h$. Let $\mathcal{N}(i)$ be the set of edges other than $i$ itself in the star of $i$ . Thus $\mathcal{N}(i)$ is the set of edges other than $i$ that are in triangles containing $i$. Note that $\langle\omega_i,\omega_j\rangle_F \ne 0$ if and only $j\in \mathcal{N}(i)$. Thus
    \begin{equation*}
        \begin{aligned}
            \norm{\alpha_h}_F^2 &= \big\langle \sum_i \alpha^i_h \omega_i^h, \sum_j \alpha^j_h \omega_j^h \big\rangle_F 
            = \sum_{i,j} \alpha^i_h \alpha^j_h \langle  \omega_i^h,  \omega_j^h \rangle_F 
            = \sum_i \sum_{j \in \mathcal{N}(i)} \alpha^i_h \alpha^j_h \langle  \omega_i^h,  \omega_j^h \rangle_F\, .
        \end{aligned}
    \end{equation*}
    Then by arithmetic mean geometric mean inequality
    \[
    \begin{aligned}
            \norm{\alpha_h}_F^2\le \sum_i \frac{1}{2}
            \bigg(\sum_{j\in \mathcal{N}(i)} \norm{\alpha^i_h \omega_i^h}_F^2 +\norm{\alpha^j_h \omega_j^h}_F^2\bigg) 
            \le 3\sum_i \norm{\alpha^i_h \omega_i^h}_F^2 
            \le \frac{3}{c_1} \norm{\alpha^i_h \omega_i^h}_D^2 \, .
        \end{aligned}  
    \]
The second last inequality holds since $X_h$ is a triangulated surface and thus for each edge $i$ there are at most 4 other edges in $\mathcal{N}(i)$ (2 if $i$ is a boundary edge and 4 otherwise). By updating $c_1$ from previous lemma, we have $\norm{\alpha_h}_D^2 \ge c_1\norm{\alpha_h}_F^2$ for any $\alpha_h \in \whitney{1}{X_h}$. 
\end{proof}

\begin{lemma}
    Let $\Omega$ be as in Lemma~\ref{lem:basis1form-delaunay} and $\{X_h\}$ a shape regular family triangulating $\Omega$. Then for any $\alpha_h \in \whitney{1}{X_h}$ we have  $\norm{\alpha_h}_D^2 \le c_2\norm{\alpha_h}_F^2$, where $c_2$ only depends on $\delta_0$.
\end{lemma}

\begin{proof}
    The proof is the same as the acute case. For each triangle, $|\;\norm{\alpha_h}_D^2\,| \le c_2 \norm{\alpha_h}_F^2$ only relies on shape regularity. The use of absolute value here allows the DEC norm to be negative on single obtuse triangle only for intermediate calculation.
\end{proof}

Based on the above results, we have that
\begin{theorem}
    Let $\Omega$ be a piecewise linear triangulated surface and $\{X_h\}$ a shape regular, uniformly Delaunay, and uniformly boundary acute family triangulating $\Omega$. Then for all $\alpha_h \in \whitney{1}{X_h}$ there exist constants $c_1, c_2$ independent of $h$ such that
\[
   c_1\norm{\alpha_h}^2_F \le \norm{\alpha_h}^2_D  \le c_2\norm{\alpha_h}^2_F\, .
\]
\end{theorem}

\subsection{Norm of Whitney 0-forms and 2-forms}
\begin{figure}[ht]
    \centering
    \begin{subfigure}{0.45\textwidth}
        \centering
        \input{figs/0-form-Delaunay-internal}
        \caption{}
        \label{fig:0form-delaunay-internal-vertex}
    \end{subfigure}
    \hfill
    \begin{subfigure}{0.45\textwidth}
        \centering
        \input{figs/0-form-Delaunay-boundary}
        \caption{}
        \label{fig:0form-Delaunay-boundary-vertex}
    \end{subfigure}
    \caption{Two examples of meshes (A) $v_0$ internal; and (B) $v_0$ on the boundary.}
    \label{fig:0form-comparison}
\end{figure}
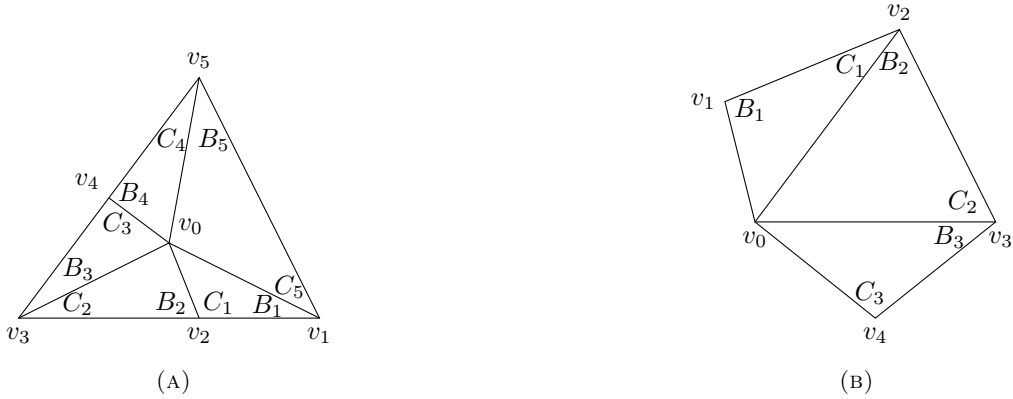
For Whitney 2-forms (piecewise constants) the DEC norm and the FEEC norm coincide. However, the situation is more subtle for 0-forms (continuous Lagrange $\mathcal{P}_1$ elements), especially in the presence of obtuse triangles.

In meshes containing obtuse triangles, the ratio between the dual cell area and the corresponding primal triangle area may no longer lie within $ \left[ \frac{1}{4}, \frac{1}{2} \right] $, which holds for acute triangulations. Nevertheless, under the assumption of the shape regularity we can still establish an upper bound on this ratio for each  triangle. This suffices to prove one direction of the norm equivalence for 0-forms.
To establish the other direction of norm equivalence, we must analyze the behavior of the basis functions across the entire mesh, analogous to the approach used for 1-forms. 

We also impose a condition we call curvature boundedness on the mesh, ensuring that the total angle sum at each vertex does not exceed $N + 2\pi$. This condition guarantees that the valence (i.e., the number of incident triangles) at each vertex remains uniformly bounded across the family of meshes, allowing us to handle non-planar geometries in a controlled manner. This condition is only needed to show that the 0-form DEC norm is bounded below for FEEC norm. It is not needed for the 1-form case.

Let $\St(v_i)$ be the star of vertex $v_i$, that is, the set of all triangles in $X_h$ that contain vertex $v_i$ as one of their vertices. Let $\mu_i^T$ be the dual area of vertex $v_i$ in triangle $T$, and $\mu^T$ the area of triangle $T$.

\begin{lemma}
    Let $T$ be a triangle with angles $A, B, C \ge \delta_0 > 0$, and edge lengths by $a, b, c$. Then there exists a constant $M > 0$ only depending on $\delta_0$, such that $\mu_A^T/\mu^T \le M$.
\end{lemma}
\begin{proof}
    By a calculation, we have $\mu_A^T = \frac{b^2}{8\tan B} + \frac{c^2}{8\tan C}$. This is the signed dual area of vertex corresponding to angle $A$ in $T$. See~\cite{HiKaVa2013} for an explanation of signed dual area. Then
    \begin{equation*}
        \frac{|\mu_A^T|}{\mu^T} = 
        \frac{|\frac{b^2}{8\tan B} + \frac{c^2}{8\tan C}|}{2R^2\sin A\sin B \sin C} = 
        \frac{|\sin B\cos B + \sin C \cos C|}{4\sin A\sin B \sin C} 
        \le \frac{2}{4\sin^3 \delta_0}\, .
    \end{equation*}
\end{proof}

\begin{remark}
    This upper bound gives us that $|\,\norm{\alpha}_D^2\,| \le M \,\norm{\alpha}_F^2$ for $\alpha \in \whitney{0}{T}$. Therefore, we have $\norm{\alpha_h}_D^2 \le M \norm{\alpha_h}_F^2$ for $\alpha_h \in \whitney{0}{X_h}$.
\end{remark}

\begin{remark}
    Before we prove other direction, note that one implication of shape regularity is that the ratio of area of two adjacent triangles has an upper bound $M_{\delta_0}$. Using Figure~\ref{fig:delaunay},
    \begin{equation*}
        \frac{\mu_{012}}{\mu_{013}} = 
        \frac{cb_1\sin A_1}{cb_2\sin A_2} = 
        \frac{(c^2/\sin C_1) \sin B_1 \sin A_1}{(c^2/ \sin C_2) \sin B_2 \sin A_2} = 
        \frac{\sin C_2 \sin B_1 \sin A_1}{\sin C_1 \sin B_2 \sin A_2} 
        \le \frac{1}{\sin^3 \delta_0}\, .
    \end{equation*}
Choosing $M_{\delta_0} = 1/(\sin^3 \delta_0)$ gives the conclusion.
\end{remark}

\begin{lemma}
    Let $v_0$ be a vertex in the mesh $X_h$, and $\lambda_0^h$ be the basis 0-form associated with $v_0$. Then there exists a constant $\varepsilon > 0$ only depending on $\delta_0, \delta_\pi, \delta_{\pi/2}$ and $N$, such that $\norm{\lambda_0^h}_D^2 \ge \varepsilon \norm{\lambda_0^h}^2_F$.
\end{lemma}

\begin{proof}
    Let $n$ be the number of triangles in $\St(v_0)$. We label the triangles clockwise in $\St(v_0)$ to be $T_i$, such that $T_i$ and $T_{i+1}$ are adjacent. (If they surround $v_0$, then $T_n$ and $T_1$ are adjacent.) The angles of $T_i$ are $A_i, B_i, C_i$, where $A_i$ is the angle at $v_0$, and the opposite edge lengths are $a_i, b_i, c_i$. Then we get that dual area of $v_0$, which is $\norm{\lambda_0^h}^2 = \sum_{i=1}^n (\frac{b_i^2}{8\tan B_i} + \frac{c_i^2}{8\tan C_i})$.
    There are two cases to consider for $v_0$. As shown in Figure \ref{fig:0form-comparison}.
    
    \noindent\textbf{Case 1:} $v_0$ is a boundary vertex. In this case, we observe that $b_i = c_{i+1}$, and $B_i$ and $C_{i+1}$ are the angles opposite to this shared edge in the adjacent triangles $T_i$ and $T_{i+1}$, respectively. Moreover, the boundary angles $C_1$ and $B_n$ are less than or equal $\frac{\pi}{2} - \delta_{\pi/2}$, since their corresponding edges lie on the boundary. Then
    \[
    \norm{\lambda_0^h}^2_D = \sum_{i=1}^n \left( \frac{b_i^2}{8\tan B_i} + \frac{c_i^2}{8\tan C_i} \right)
    = \frac{c_1^2}{8\tan C_1} + \frac{b_n^2}{8\tan B_n} + \sum_{i=1}^{n-1} \left( \frac{b_i^2}{8\tan B_i} + \frac{b_i^2}{8\tan C_{i+1}} \right)\,.
    \]
    This gives the inequality
    \[
    \norm{\lambda_0^h}^2_D \ge \frac{c_1^2}{8\tan C_1} + \frac{b_n^2}{8\tan B_n}\,.
    \]
    since $B_i + C_{i+1} < \pi$ and so $(1/\tan B_i) + (1/\tan C_{i+1}) > 0$. We also have 
    \[
    \frac{c_1^2}{8\tan C_1} = \frac{R_1^2 \sin C_1 \cos C_1}{2} \ge \frac{\mu^{T_1}\sin 2C_1}{8} \ge \min\left\{\frac{\sin 2\delta_0}{8}, \frac{\sin2\delta_{\pi/2}}{8}\right\}\,\mu^{T_1}
    \]
    Then we have $\norm{\lambda_0}^2_D  \ge \min\{(\sin 2\delta_0)/8, (\sin2\delta_{\pi/2})/8\}\,\mu^{T_1}$.

    \medskip
    \noindent\textbf{Case 2:} $v_0$ is an interior vertex. In this case, the triangle $T_n$ is adjacent to $T_1$, forming a cyclic patch of $n$ triangles around $v_0$. Let $i$ be the index such that the edge $v_0v_i$ is the longest among all edges incident to $v_0$. We claim that either the pair of angles $(B_i, C_i)$ or $(B_{i+1}, C_{i+1})$ must both be acute.
    
    To see this, note that since $b_i > c_i$ and $b_{i+1} < c_{i+1}$, we have $B_i \ge C_i$ and $C_{i+1} > B_{i+1}$. Moreover, because $B_i + C_{i+1} \le \pi$, at least one of these angles must be acute. Once we identify such a triangle with both adjacent angles acute, its third angle must also be acute (since the sum of the angles in a triangle is $\pi$), implying that this triangle is acute.
    
    It follows that among the $2n$ angles appearing in the DEC norm expression, at most $2n - 2$ of them can be obtuse. Moreover, every obtuse angle appears in a pair of adjacent triangles, and the contribution from each such pair can be canceled using the opposite angle like in case 1. Thus, there are at least two acute angles that remain uncanceled in the summation of the DEC norm, which provides the required lower bound. In summary, there exists a $j$, such that 
    \[
    \norm{\lambda_0^h}^2_D  \ge 
    \min\bigg\{\frac{\sin 2\delta_0}{8}, \frac{\sin2\delta_\pi}{8}, \frac{\sin2\delta_{\pi/2}}{8}\bigg\}\mu^{T_j}\, .
    \]
    And for the FEEC norm of $\lambda_0$ we note that
    \[
    \norm{\lambda_0^h}^2_F = \frac{1}{6}\sum_{i=1}^n \mu^{T_i} \le \frac{1}{6}\sum_{i=1}^n M_{\delta_0}^i \mu^{T_j} \le \hat{M}_{\delta_0} \mu^{T_j}\, .
    \]
    This is true because $n \le (2\pi+N)/\delta_0$. Therefore, we have $\norm{\lambda_0^h}_D^2 \ge \varepsilon \norm{\lambda_0^h}^2_F$. By choosing $\varepsilon = \min\{\frac{\sin 2\delta_0}{8\hat{M_{\delta_0}}}, \frac{\sin2\delta_\pi}{8\hat{M_{\delta_0}}}\}$.    
\end{proof}

\begin{theorem} \label{thm:delaunay-equivalence}
    Let $\Omega$ be a piecewise linear triangle mesh surface and $\{X_h\}$ a shape regular, uniformly Delaunay, uniformly boundary acute and curvature bounded family of meshes triangulating $\Omega$. Then for all $\alpha_h \in \whitney{0}{X_h}$ there exist constants $c_1, c_2$ independent of $h$ such that
    \[
   c_1\norm{\alpha_h}^2_F \le \norm{\alpha_h}^2_D  \le c_2\norm{\alpha_h}^2_F\, .
    \]
\end{theorem}
\begin{proof}
    We have shown that shape regularity implies $\norm{\alpha_h}^2_D  \le c_2\norm{\alpha_h}^2_F$. For the other inequality let $\alpha_h = \alpha_h^i \lambda_i^h$ and let $\mathcal{N}(i)$ be the set of vertices that are neighbors of $v_i$. These are all the vertices connected to $v_i$ by an edge. Then
    \begin{equation*}
        \begin{aligned}
            \norm{\alpha_h}_F^2 &= 
            \big\langle \sum_i \alpha^i_h \lambda_i^h, \sum_j \alpha^j_h \lambda_j^h 
            \big\rangle_F = 
            \sum_{i,j} \alpha^i_h \alpha^j_h \langle  \lambda_i^h,  \lambda_j^h \rangle_F = 
            \sum_{i} \sum_{j\in \mathcal{N}(i)}\alpha^i_h \alpha^j_h \langle  \lambda_i^h,  \lambda_j^h \rangle_F \\
            &\le \frac{1}{2}\sum_{i}\sum_{j\in\mathcal{N}(i)} 
            \norm{\alpha_h^i\lambda_i^h}^2_F + \norm{\alpha_h^j\lambda_j^h}^2_F\, .
        \end{aligned}
    \end{equation*}
    By the curvature boundedness and shape regularity there are at most $(2\pi+N)/\delta_0$ edges connecting $i$. Thus
    \begin{equation*}
        \norm{\alpha_h}_F^2 \le \bigg(1 + \frac{\pi+N}{\delta_0}\bigg)\sum_{i} \norm{\alpha_h^i\lambda_i^h}^2_F
            \le \bigg(\big(1 + \frac{\pi+N}{\delta_0}\big)/\varepsilon\bigg) \norm{\alpha_h}_D^2\, .
    \end{equation*}
    Thus $c_1\norm{\alpha_h}^2_F \le \norm{\alpha_h}^2_D  \le c_2\norm{\alpha_h}^2_F$ for any $\alpha_h \in \whitney{0}{X_h}$.
\end{proof}
\begin{remark}
    With norm equivalence, the stability and convergence are a direct consequence of the theoretical framework for Hilbert complexes developed by Holst and Stern \cite{Holst_2012}. We can cast both the DEC and FEEC formulations as Hilbert complexes where the identity map serves as the required injective morphism. This condition is met for uniformly Delaunay and shape-regular meshes due to established norm equivalences, thus guaranteeing stability and convergence.
\end{remark}

\section{DEC and FEEC inner products for DEC regular triangulations}\label{sec:ip}

In this section we continue to assume that $\Omega$ is a piecewise linear triangle mesh surface. From here on we will assume that the family of meshes $\{X_h\}$ triangulating $\Omega$ is shape regular and non-degenerate Delaunay. That is, the family $\{X_h\}$ is DEC regular. In the absence of uniform Delaunay condition we will not have norm equivalence of DEC norm and FEEC norm, but shape regularity is enough to show $\norm{\alpha}_D \le c_2 \norm{\alpha}_F$ for Whitney 1-forms $\alpha$ as long as DEC $*_1$ is positive definite, which just requires non-degenerate Delaunay. This is sufficient for comparing DEC and FEEC inner products in this section and for stability and convergence in~\S\ref{sec:convergence}.

We will use the Hilbert complexes notation of~\cite{ArFaWi2010}. That is, we have a Hilbert complex $(W,d)$ and a bounded Hilbert complex $(V,d)$ called the domain complex~\cite[page 301]{ArFaWi2010}. The spaces $V^k$ are the domains of $d^k: V^k \to V^{k+1}$ and are endowed with the inner product associated to the graph norm~\cite[page 301]{ArFaWi2010} as
\begin{equation}\label{eq:FEEC-graph-ip}
    \langle u_h, v_h \rangle_{V^k} = \langle u_h, v_h \rangle_F + \langle d u_h, d v_h \rangle_F\, .
\end{equation}
Where~\cite{ArFaWi2010} write $\langle \,\rule{0.5em}{0.4pt}\,,\, \rule{0.5em}{0.4pt}\rangle_{W^k}$ or $\langle \,\rule{0.5em}{0.4pt}\,,\, \rule{0.5em}{0.4pt}\rangle$ we write $\langle \,\rule{0.5em}{0.4pt}\,,\, \rule{0.5em}{0.4pt}\rangle_F$ to contrast with the DEC version defined next. In~\cite{ArFaWi2010} the spaces $V_h^k$ are finite dimensional subspaces of $V^k$. In this paper we are only interested in $V^k = H\Lambda^k(\Omega)$ and $V_h^k = \whitney{k}{X_h}$ the space of lowest degree Whitney forms. The next definition is the DEC analogue of~\eqref{eq:FEEC-graph-ip}.
\begin{definition}
    For $u_h, v_h \in  V_h^k = \whitney{k}{X_h}$ the space $V_h^k$ is endowed with the inner product associated with the graph norm for DEC as
    \[  
    \langle u_h, v_h \rangle_{V_h^k}^D := \langle u_h, v_h \rangle_D + \langle d u_h, d v_h \rangle_D\, .
    \]
\end{definition}    


The next property we aim to establish is the consistency of the DEC inner product with respect to the FEEC inner product. We remark that a Whitney form on a simplex is constant if and only if it is closed. The technique is to check, simplex by simplex, that the two inner products coincide if both arguments are constant.


A version of Lemma~\ref{lem:ip0} below has been known in the finite element literature for the standard mass lumping for the mass matrix of continuous $\mathcal{P}_1$ Lagrange finite elements. In that mass lumping technique the off diagonal entries are added to the diagonal and replaced by 0. The resulting diagonal mass matrix approximation is different from the corresponding DEC mass matrix for 0-forms $*_0$ which is used in Lemma~\ref{lem:ip0} below.
\begin{lemma}\label{lem:ip0}
    Let $\alpha, \beta \in \whitney{0}{T}$. Then 
    \[
       | \langle \alpha, \beta \rangle_D - \langle \alpha, \beta \rangle_F | \le Ch\norm{\alpha}_V \norm{\beta}_V\, ,
    \]
    where $h$ is the circumradius of $T$, and $C$ is a constant independent of the shape or $h$.
\end{lemma}

\begin{proof}
    We will consider three cases:
    \begin{enumerate}
        \item both $\alpha, \beta$ are constant;
        \item one of $\alpha, \beta$ is constant; and
        \item neither of $\alpha, \beta$ is constant.
    \end{enumerate}
    Let the vertices of triangle $T$ be $v_0, v_1, v_2$, angles $A, B, C$, dual areas are $\mu_0, \mu_1, \mu_2$, and triangle area $\mu$. As shown in previous section, we know that $\mu_0 + \mu_1 + \mu_2 = \mu,$ and 
    $|\mu_i/\mu|$ is bounded. Let $\alpha, \beta$ on vertex $v_i$ be $\alpha_i, \beta_i$, and $\lambda_i$ the barycentric coordinate associated with $v_i$.

    \noindent\textbf{Case 1}: In this case $\langle \alpha, \beta \rangle_D = \langle \alpha, \beta \rangle_F$. This is because both are equal to $\alpha_0\beta_0\mu$ since
    \[
    \langle \alpha, \beta \rangle_D = \sum_i \alpha_i\beta_i \mu_i = \alpha_0\beta_0\sum_i \mu_i = \alpha_0\beta_0\mu\, ,
    \]
    and
    \[
    \langle \alpha, \beta \rangle_F = \int_T (\alpha_0\lambda_0 +  \alpha_1\lambda_1 + \alpha_2\lambda_2)( \beta_0 \lambda_0+\beta_1 \lambda_1+\beta_2 \lambda_2)= \int_T \alpha_0 \beta_0 = \alpha_0 \beta_0 \mu\, .
    \]
    
    \noindent\textbf{Case 2}: Assume $\alpha$ is constant. Then $\langle \alpha, \beta \rangle_F = \int_T \alpha_0(\beta_0\lambda_0 + \beta_1\lambda_1 + \beta_2\lambda_2) = \frac{\alpha_0\mu}{3}(\beta_0 + \beta_1 + \beta_2)$, and $\langle \alpha, \beta \rangle_D = \alpha_0(\beta_0\mu_0 + \beta_1\mu_1 + \beta_2\mu_2)$. Thus 
    \begin{equation*}
        \begin{aligned}
            | \langle \alpha, \beta \rangle_D - \langle \alpha, \beta \rangle_F | &=
            \big| \alpha_0\mu\sum_i 
            \beta_i\big(\frac{1}{3} - \frac{\mu_i}{\mu}\big)\big| 
            = \mu\;\; |\alpha_0|\; 
            \big|\big((\beta_1-\beta_0)\big(\frac{1}{3} - \frac{\mu_1}{\mu}\big) +
            (\beta_2-\beta_0)\big(\frac{1}{3}- \frac{\mu_2}{\mu}\big)\big)\big| \\
            &= \sqrt{\mu}\;\; \norm{\alpha}_F\;\;  
            \big|((\beta_1-\beta_0)(\frac{1}{3} - \frac{\mu_1}{\mu}) + 
            (\beta_2-\beta_0)(\frac{1}{3} - \frac{\mu_2}{\mu}))\big|\, .
        \end{aligned}
    \end{equation*}
    Then since $|\mu_i/\mu|$ is bounded and FEEC norm is greater than $l_2$ norm for $1$-forms. (Note that the only geometric condition required here is shape regularity.)
    \[
    | \langle \alpha, \beta \rangle_D - \langle \alpha, \beta \rangle_F |
    \le \sqrt{2h^2} \;\norm{\alpha}_F \;C\,\norm{d\beta}_F 
    \le \hat{C} h\, \norm{\alpha}_F\; \norm{d\beta}_F\, .
    \]
    \noindent\textbf{Case 3} As in Case 2, we use $f(\alpha, \beta)$ to denote $| \langle \alpha, \beta \rangle_D - \langle \alpha, \beta \rangle_F |$. Then
    \[
    f(\alpha, \beta) \le f(\alpha-\alpha_0, \beta-\beta_0) + f(\alpha-\alpha_0, \beta_0) + f(\alpha_0, \beta)\, .
    \]
    Since $f(\alpha-\alpha_0, \beta_0)$, $f(\alpha_0, \beta)$ are Case 2,  we  have $f(\alpha-\alpha_0, \beta_0) \le \hat{C}\,h\,\norm{d\alpha}_D\norm{\beta}_D$ and $f(\alpha_0, \beta) \le \hat{C}\,h\,\norm{\alpha}_D\norm{d\beta}_D$. Now we estimate $f(\alpha-\alpha_0, \beta-\beta_0)$, similar to how we deal with $\beta-\beta_0$ in Case 2. We can bound this by $\hat{C}\,h\,\,\norm{d\beta}_D\;\norm{d\alpha}_D$. 
    In summary, for any $\alpha, \beta \in \whitney{0}{T}$ we have $| \langle \alpha, \beta \rangle_D - \langle \alpha, \beta \rangle_F | \le Ch\,\norm{\alpha}_V\;\norm{\beta}_V$.
\end{proof}

The proof for Lemma~\ref{lem:ip1} below is essentially identical to the proof in~\cite{ChHa2011}. 
\begin{lemma} \label{lem:ip1}
    Let $\alpha, \beta \in \whitney{1}{T}$. Then
    \[
       | \langle \alpha, \beta \rangle_D - \langle \alpha, \beta \rangle_F | \le Ch\norm{\alpha}_V\norm{\beta}_V
    \]
    where $h$ is the circumradius of the triangle $T$.
\end{lemma}

\begin{proof}
    The proof is similar to last lemma. We consider three cases:
    \begin{enumerate}
        \item both $\alpha, \beta$ are closed, which means they are exact since it is only one triangle;
        \item one of $\alpha, \beta$ is exact; and
        \item neither of $\alpha, \beta$ is exact.
    \end{enumerate}
    \noindent\textbf{Case 1}: Let $\alpha = d(f_0\lambda_0 + f_1\lambda_1+f_2\lambda_2), \beta = d(g_0\lambda_0 + g_1\lambda_1+g_2\lambda_2)$, then 
    \begin{equation*}
        \begin{aligned}
            \langle \alpha, \beta \rangle_F = 
            \int_T (f_0d\lambda_0 +f_1d\lambda_1+f_2d\lambda_2, g_0d\lambda_0 +g_1d\lambda_1+g_2d\lambda_2) = 
            \int_T \sum_{i,j}f_ig_j\langle d\lambda_i d\lambda_j \rangle 
            = \mu \sum_{i,j} f_ig_jh_{ij}\, ,
        \end{aligned}
    \end{equation*}
    and
    \begin{equation*}
        \begin{aligned}
            \langle \alpha, \beta \rangle_D  &= -(f_1 - f_0)(g_1-g_0)\mu h_{01} 
            -(f_2 - f_1)(g_2-g_1)\mu h_{12} -(f_0 - f_2)(g_0-g_2)\mu h_{20} \\
            &= \sum_{i\ne j}f_ig_j\mu h_{ij} - f_0g_0\mu(h_{01}+h_{02}) - f_1g_1\mu(h_{01}+h_{12}) - f_2g_2\mu(h_{02}+h_{12})
        \end{aligned}
    \end{equation*}
    To show that $\langle \alpha, \beta \rangle_F = \langle \alpha, \beta \rangle_D$, we only need to show that $h_{01}+h_{02} = -h_{00}, h_{01}+h_{12} = -h_{11}, h_{02}+h_{12} = -h_{22}$. This is because $\langle d\lambda_i, d\lambda_0 + d\lambda_1 +d\lambda_2\rangle = \langle d\lambda_i, d1\rangle = 0$.
    
    \noindent\textbf{Case 2}. Assume $\alpha$ is closed, and $\beta$ is not. Let $\alpha = (f_0d\lambda_0 + f_1d\lambda_1+f_2d\lambda_2)$, and $\beta = g_2 (\lambda_0d\lambda_1 - \lambda_1d\lambda_0)$, we can choose such beta because we can minus a closed 1-form $\bar{\beta}$, such that $\beta-\bar{\beta}$ takes $0$ on two other edge. ($\langle \alpha, \beta \rangle_F - \langle \alpha, \beta \rangle_D = \langle \alpha, \beta-\bar{\beta} \rangle_F - \langle \alpha, \beta-\bar{\beta} \rangle_D$).
    Now, we have:
    \begin{equation*}
        \begin{aligned}
            |\langle \alpha, \beta \rangle_F| = |g_2\langle \alpha, \omega_2 \rangle_F | 
            \le \bigg(\int_Td\beta\bigg) \norm{\alpha}_F \; \norm{\omega_2}_F 
            \le C'\sqrt{\mu}\; \norm{d\beta}_F \; \norm{\alpha}_F 
            \le C'\sqrt{2}h\; \norm{d\beta}_F\;\norm{\alpha}_F\,,
        \end{aligned}
    \end{equation*}
    and
    \begin{equation*}
        \begin{aligned}
            |\langle \alpha, \beta \rangle_D| = |g_2(f_1-f_0)h_{01}|  
            \le  \norm{\alpha}_F \; \;\bigg|\int_Td\beta\bigg|  
            \le \norm{\alpha}_F\; \sqrt{2}\,h\;\norm{d\beta}_F \, .
        \end{aligned}
    \end{equation*}
    For the first inequality of DEC norm analysis, we used that $h_{01}$ is bounded, and FEEC norm is greater than $l_2$ norm for $1$ form. Therefore, we have 
    \[
    | \langle \alpha, \beta \rangle_F - \langle \alpha, \beta \rangle_D| \le Ch\,\norm{\alpha}_F\; \norm{d\beta}_F\, .
    \]

    \noindent\textbf{Case 3}: When neither $\alpha, \beta$ are closed, we could use the same method as in 0-form to prove it. We would have
    \[
    |\langle \alpha - \bar{\alpha}, \beta-\bar{\beta} \rangle_F - \langle \alpha - \bar{\alpha}, \beta-\bar{\beta} \rangle_D| 
    \le Ch\,\norm{d\alpha}_F\; \norm{d\beta}_F\, .
    \]
    Thus for any $\alpha, \beta \in \whitney{1}{T}$, 
    \[
    | \langle \alpha, \beta \rangle_D - \langle \alpha, \beta \rangle_F | 
    \le C\,h\;\norm{\alpha}_V\;\norm{\beta}_V\, .
    \]
\end{proof}

\begin{theorem}\label{thm:ip-error}
    Given a family $\{X_h\}$ of DEC-regular triangulations of $\Omega$, let $\alpha_h, \beta_h \in \whitney{k}{X_h}$ for a mesh $X_h$. Then
    \[
    |\langle \alpha_h, \beta_h \rangle_D - \langle \alpha_h, \beta_h \rangle_F | 
    \le Ch\,\norm{\alpha_h}_{V} \;\norm{\beta_h}_{V}\, .
    \] 
\end{theorem}

\begin{proof}
    Using Lemmas~\ref{lem:ip0} and~\ref{lem:ip1} and the fact that for Whitney 2-forms, the DEC inner product and FEEC inner product are the same we have
    \begin{equation*}
        \begin{aligned}
            \big| \langle \alpha_h, \beta_h \rangle_D - 
            \langle \alpha_h, \beta_h \rangle_F \big| &\le 
            \sum_{T\in X_h} \big|\langle \alpha_h|_T, \beta_h|_T\rangle_D - 
            \langle \alpha_h|_T, \beta_h|_T \rangle_F\big| \le 
            C \,h \sum_T\norm{\alpha_h|_T}_V\;\norm{\beta_h|_T}_V \\
            &\le C\,h\,\big(\sum_T\norm{\alpha_h|_T}_V^2\big)^{1/2}\;
            \big(\sum_T\norm{\beta_h|_T}_V^2\big)^{1/2}\le 
            C\,h \,\norm{\alpha_h}_V\;\norm{\beta_h}_V.
        \end{aligned}
    \end{equation*}
\end{proof}

\begin{remark}
    In fact, we can have a slightly stronger version of this theorem: $|\langle \alpha_h, \beta_h \rangle_D - \langle \alpha_h, \beta_h \rangle_F | 
    \le Ch\,(\norm{\alpha_h}_F \norm{d\beta_h}_F + \norm{d\alpha_h}_F \norm{\beta_h}_F) $, by modifying the conclusion of Lemmas~\ref{lem:ip0} and~\ref{lem:ip1} to the same format, which is evident in the proof.
\end{remark}

\section{Stability and Convergence of DEC}\label{sec:convergence}

For stability and convergence proof of DEC method for~\eqref{eq:Hodge-Laplace-natural} we follow~\cite{ArFaWi2010}. But first, with the results from previous section, we can show a DEC version of Poincar\'e inequality as follows.
\begin{theorem}[DEC Poincar\'e Inequality]\label{thm:PI}
     Given a family $\{X_h\}$ of DEC-regular triangulations of $\Omega$ a piecewise linear triangle mesh surface, we have
     \[
     \norm{v}_V \le \hat{c}_P \norm{dv}_V \quad \text{and} \quad \norm{v}_{V_h}^D \le \hat{c}_P \norm{dv}_{V_h}^D \quad \forall v \in \Z_h^{\perp_D}\, ,
     \]
     where $\hat{c}_P$ is a constant. 
\end{theorem}

\begin{proof}
    We know that $\Z_h^{\perp_D}$ is different from $\Z_h^{\perp}$, so we assume that $ v = v^{\perp} + v_0$, where $v^{\perp} \in \Z_h^{\perp}$ and $v_0 \in  \Z_h$.
    Using the previous remark, we have 
    \[
    |\langle v, \alpha \rangle_D - \langle v, \alpha \rangle_F| \le Ch(\norm{v}_F\norm{d\alpha}_F + \norm{dv}_F\norm{\alpha}_F) \qquad \forall \alpha \in V_h\, .
    \]
    Taking $\alpha = v_0$
    \[
        \norm{v_0}^2_F \le C\,h\norm{dv^{\perp}}_F\;\norm{v_0}_F.
    \]
    Thus $\norm{v_0}_V \le C\,h\,\norm{dv^{\perp}}_F$ and $\norm{v}_V \le \norm{v_0}_V + \norm{v^{\perp}}_V \le C\,h\,\norm{dv^{\perp}}_F + c_P\,\norm{dv^{\perp}}_V \le (C\,h\, + c_P) \norm{dv^{\perp}}_V$ . Here we have used the FEEC Poincar\'e inequality and the FEEC Poincar\'e constant $c_P$.  Therefore 
    \[
    \norm{v}_V \le (C\,h\, + c_P)\norm{dv}_V\, .
    \]
    Which proves the inequality for FEEC graph norms but using DEC inner product for orthogonal complement. For the result using DEC graph norms we have $\norm{v}_D \le c_2 \norm{v}_F$ and $\norm{dv}_V = \norm{dv}_{V_h}^D$ (This is using the Lemmas~\ref{lem:ip0} and~\ref{lem:ip1} using Case 1) and thus $\norm{v}_{V_h}^D \le \hat{c}_P \norm{dv}_{V_h}^D$. 
\end{proof}

Since FEEC results for two dimensions in~\cite{ArFaWi2010} apply to planar domains from this point on we will assume $\Omega\subset \R^2$. The family $\{X_h\}$ triangulating $\Omega$ will be assumed to be DEC-regular, that is, non-degenerate Delaunay and shape regular. 
Following~\cite{ArFaWi2010} the mixed formulation of the abstract Hodge-Laplace problem is to find $(\sigma, u, p) \in V^{k-1} \times V^k \times \H^k$ satisfying:
\begin{equation}
    \begin{aligned}
        \langle \sigma, \tau \rangle - \langle d\tau,  u \rangle &= 0,  &&\tau \in V^{k-1},  \\
        \langle d\sigma, v \rangle + \langle du,  dv \rangle + \langle v, p \rangle & = \langle f, v \rangle, && v \in V^k, \\
         \langle u, q \rangle & = 0, && q\in \H^k,
    \end{aligned}
\end{equation}

Following~\cite[\S3.2.2]{ArFaWi2010} define a bilinear form in terms of DEC inner product as
\begin{equation}\label{eq:DEC-bilinear}
    B_D(\sigma, u, p;\tau, v, q) = \langle\sigma , \tau \rangle_D - \langle d\tau, u \rangle_D + \langle d\sigma, v \rangle_D + \langle du, dv \rangle_D + \langle v, p \rangle_D\ - \langle u, q \rangle_D\, .
\end{equation}
This is the DEC version of the FEEC bilinear form $B$ from~\cite[\S3.2.2]{ArFaWi2010} which we will refer to as $B_F$. Thus
\[
B_F(\sigma, u, p;\tau, v, q) = \langle\sigma , \tau \rangle_F - \langle d\tau, u \rangle_F + \langle d\sigma, v \rangle_F + \langle du, dv \rangle_F + \langle v, p \rangle_F\ - \langle u, q \rangle_F\, .
\]


As in FEEC we aim to bound the discretization error in terms of the stability of the discretization and the consistency error of DEC method. Our first step is to establish a lower bound on the inf-sup constant, which equivalently provides an upper bound on the stability constant. This ensures the well-posedness of the discrete problem and allows for derivation of error estimates.

\begin{theorem}[DEC stability using FEEC norm and FEEC harmonic forms]\label{thm:inf-sup-FEEC-norm}
    Let $\{X_h\}$ be a family of DEC-regular triangulations of $\Omega$. Then there exists a $\gamma > 0$, depending only on $\delta_0$ and $c_P$ in FEEC Poincar\'e inequality, such that for any $(\sigma, u, p) \in \whitney{k-1}{X_h} \times \whitney{k}{X_h} \times \H_h^k$, when mesh size $h$ is small enough, there exists $(\tau, v, q) \in \whitney{k-1}{X_h} \times \whitney{k}{X_h} \times \H_h^k$ with
    \[
    B_D(\sigma, u, p;\tau, v, q) \ge \gamma (\norm{\sigma}_{V} + \norm{u}_{V} + \norm{p})(\norm{\tau}_{V}+\norm{v}_{V} + \norm{q})
    \]
\end{theorem}
\begin{proof}
    First following~\cite{ArFaWi2010}, use Hodge decomposition, $u = u_{\B_h} + u_{\perp} + u_{\H_h}$, where $u_{\B_h} = P_{\B_h}u$ and $u_{\perp} = P_{\Z_h^{k\perp}}u$, $u_{\H_h} = P_{\H_h}u$.
    Then $u_{\B_h} = d\rho$, where $\rho \in \Z_h^{k-1\perp}$ and thus
    \[\norm{\rho}_V \le c_P \norm{u_{\B_h}}_V, \qquad \norm{u_{\perp}}_V \le c_P\norm{du}_V , \qquad c_P > 1.
    \]
    Choosing $\tau = \sigma - \frac{1}{c_P^2}\rho$, $v = u + d\sigma + p$ and $q= p - u_{\H_h}$ as in~\cite[Theorem 3.2]{ArFaWi2010}
    \[
    (\norm{\tau}_V + \norm{v}_V + \norm{q}_F) \le C(\norm{\sigma}_V + \norm{u}_V+\norm{p}_F)\, .
    \]
    Using the above values of $\tau, v, q$ in definition~\eqref{eq:DEC-bilinear} of $B_D$
    \begin{align*}
        B_D(\sigma, u, p;\tau, v, q)  = &\langle\sigma , \tau \rangle_D - \langle d\tau, u \rangle_D + \langle d\sigma, v \rangle_D + \langle du, dv \rangle_D +\langle v, p \rangle_D\ - \langle u, q \rangle_D \\
            = & \norm{\sigma}^2_D - \frac{1}{c_P^2}\langle\sigma , \rho \rangle_D + \frac{1}{c_P^2}\langle u_{\B_h} , u \rangle_D + \norm{d\sigma}^2_D + \norm{du}^2_D + 2\langle d\sigma, p\rangle_D + \norm{p}_D^2 + \langle u, u_{\H_h} \rangle_D \, .
    \end{align*}
    Here we used the fact that $d\rho = u_{\B_h}$ and $\langle du, dv \rangle_D = \langle du, du\rangle_D$. Then since 
    \begin{equation*}
        \begin{aligned}
            B_D(\sigma, u, p;\tau, v, q)  = &\langle\sigma , \tau \rangle_D - \langle d\tau, u \rangle_D + \langle d\sigma, v \rangle_D + \langle du, dv \rangle_D +\langle v, p \rangle_D\ - \langle u, q \rangle_D \\
            = & \norm{\sigma}^2_D - \frac{1}{c_P^2}\langle\sigma , \rho \rangle_D + \frac{1}{c_P^2}\langle u_{\B_h} , u \rangle_D + \norm{d\sigma}^2_D + \norm{du}^2_D + 2\langle d\sigma, p\rangle_D + \norm{p}_D^2 + \langle u, u_{\H_h} \rangle_D \\
            \ge & \frac{1}{2}\norm{\sigma}^2_D + \norm{d\sigma}^2_D + \norm{du}^2_D - \frac{1}{2c_P^4}\norm{\rho}^2_D + \frac{1}{c_P^2}\langle u_{\B_h} , u \rangle_F - \hat{C}\,h\,\norm{u_{\B_h}}_V\norm{u}_V     \\
            & - 2\hat{C}\,h\,\norm{d\sigma}_V\norm{p}_V + \norm{p}_D^2 + \langle u, u_{\H_h} \rangle_F - \hat{C}\,h\,\norm{u_{\H_h}}_V\norm{u}_V \\
            \ge & \frac{1}{2}\norm{\sigma}^2_D + \norm{d\sigma}^2_D + \norm{du}^2_D - \frac{1}{2c_P^4}\norm{\rho}^2_D + \frac{1}{c_P^2}\norm{u_{\B_h}}^2_F - 2\hat{C}\,h\,\norm{u}_V^2  \\
            & + \norm{p}^2_D + \norm{u_{\H_h}}^2_F - \hat{C}\,h\,\norm{d\sigma}_V^2 - \hat{C}\,h\,\norm{p}_V^2\\
            \ge & \frac{1}{2}\norm{\sigma}^2_D + \norm{d\sigma}^2_D + \norm{du}^2_D + \frac{1}{2c_P^2}\norm{u_{\B_h}}^2_F - \frac{\hat{C}\,h\,}{2c_P^4}\norm{\rho}_V^2 -2\hat{C}\,h\,\norm{u}_V^2  \\
            & + \norm{p}_D^2 + \norm{u_{\H_h}}^2_F - \hat{C}\,h\,\norm{d\sigma}_V^2 - \hat{C}\,h\,\norm{p}_V^2\\
            \ge & \frac{1}{2}\norm{\sigma}^2_F + \norm{d\sigma}^2_F + \frac{1}{2}\norm{du}^2_F + \frac{1}{2c_P^2}\norm{u_{\B_h}}^2_F + \frac{1}{2c_P^2}\norm{u_{\perp}}^2_F - 4\hat{C}\,h\,\norm{u}_V^2 \\
            & + \norm{p}_F^2 + \norm{u_{\H_h}}^2_F - 2\hat{C}\,h\,\norm{d\sigma}_V^2 - 2\hat{C}\,h\,\norm{p}_V^2 - \hat{C}\,h\,(\norm{\sigma}^2_V)\\
            \ge & \frac{1}{2c_P^2} (\norm{\sigma}_V^2 + \norm{u}_V^2 + \norm{p}_V^2) - 4\hat{C}\,h\,(\norm{u}_V^2 + \norm{\sigma}_V^2 + \norm{p}_V^2)\\
        \end{aligned}
    \end{equation*}
    Letting $h \le 1/(16c_P^2\hat{C})$ proves the result.
\end{proof}


From this stability result, we obtain the following error estimate between the DEC and FEEC solutions, for which we require our DEC solution to be orthogonal to the FEEC harmonic space $\H^k$. 
\begin{theorem}[DEC error using FEEC norm and FEEC harmonic forms] \label{thm:error-FEEC-H}
    Let $\{X_h\}$ be a family of DEC-regular triangulations of $\Omega$. Let $(\hat{u}_h, \hat{\sigma}_h, \hat{p}_h) \in \whitney{k-1}{X_h} \times \whitney{k}{X_h} \times \H_h^k$ be the DEC solution and $(u_h, \sigma_h, p_h) \in \whitney{k-1}{X_h} \times \whitney{k}{X_h} \times \H_h^k$  the FEEC solution. Then we have the error estimate:
    \[
    \norm{\hat{\sigma}_h - \sigma_h}_V + \norm{\hat{u}_h - u_h}_V + \norm{\hat{p}_h - p_h}_V \le Ch\big(\norm{\sigma_h}_V + \norm{u_h}_V + \norm{p_h}_F + \norm{f_h}_V\big)\, .
    \]
\end{theorem}

\begin{proof}
    First of all, we have
    \[
    B_D(\hat{\sigma}_h, \hat{u}_h, \hat{p}_h; \tau_h, v_h, q_h) = \langle f_h, v_h\rangle_D, \quad  (\tau_h, v_h, q_h) \in \whitney{k-1}{X_h} \times \whitney{k}{X_h} \times \H_h^k
    \]
    and 
    \[
    B(\sigma_h, u_h, p_h; \tau_h, v_h, q_h) = \langle f_h, v_h\rangle_F, \quad  (\tau_h, v_h, q_h) \in \whitney{k-1}{X_h} \times \whitney{k}{X_h} \times \H_h^k
    \]
    Substituting all FEEC inner product in second equation with DEC inner product, we will have 
    \[
    |B_D(\sigma_h, u_h, p_h; \tau_h, v_h, q_h) - \langle f_h, v_h\rangle_D | \le \hat{C}\,h\, (\norm{\sigma_h}_V + \norm{u_h}_V + \norm{p_h}_F +\norm{f_h}_V)(\norm{\tau_h}_V + \norm{v_h}_V + \norm{q_h}_V)\, .
    \]
    Now we have 
    \begin{equation*}
        \begin{aligned}
            |B_D(\hat{\sigma}_h-\sigma_h, \hat{u}_h &- u_h, \hat{p}_h-p_h; \tau_h, v_h, q_h)| = |B_D(\sigma_h, u_h, p_h; \tau_h, v_h, q_h) - \langle f_h, v_h\rangle_D| \\
            &\le \hat{C}\,h\, (\norm{\sigma_h}_V + \norm{u_h}_V + \norm{p_h}_F + \norm{f_h}_V)(\norm{\tau_h}_V + \norm{v_h}_V + \norm{q_h}_F)\, .
        \end{aligned}
    \end{equation*}
    Theorem \ref{thm:inf-sup-FEEC-norm} then gives
    \[
    \norm{\hat{\sigma}_h - \sigma_h}_{V} + \norm{\hat{u}_h - u_h}_{V} + \norm{\hat{p}_h - p_h}_V\le Ch(\norm{\sigma_h}_V + \norm{u_h}_V + \norm{p_h}_F + \norm{f_h}_V)\, .
    \]
\end{proof}

The above theorems are using FEEC harmonic space instead of DEC harmonic space. But if we want to use DEC harmonic space, we need to establish the following theorem:

\begin{theorem}[DEC and FEEC harmonic forms]\label{thm:harmonic-comparison}
    Let $\hat{\H}_h$ denote the space of DEC harmonic forms, and $\H_h$ the space of FEEC harmonic forms. Then there exists a bijection $\Pi_h:\hat{\H}_h \to\H_h$ such that 
    \begin{equation}\label{eq:dec-feec-harmonic}
    \norm{\Pi_h(\hat{p}) - \hat{p}}_F \le \hat{C}\,h\,\norm{\hat{p}}_F, \qquad \forall ~\hat{p} \in \hat{\H}_h\, .    
    \end{equation}
\end{theorem}

\begin{proof}
    We first define $\Pi_h$. For any $\hat{p} \in \hat{\H}_h$, we know that $\hat{p} \in \hat{\Z}_h = \Z_h$ and so $\hat{p} = p_0 + p_{\perp}$, where $p_0 \in \H_h$, and $p_{\perp} \in \B_h = \hat{\B}_h$. Define $\Pi_h(\hat{p}) = p_0$. This is well defined because of the uniqueness of the decomposition $\hat{p} = p_0 + p_{\perp}$. 
    
    To show that this map is a bijection first note that if $\Pi_h(\hat{p}) = 0$, then by definition of $\Pi_h$, $\hat{p} \in \B_h$, which implies that $\hat{p} = 0 \in \hat{\Z}_h$. Thus $\Pi_h$ is injective. For surjectivity first note that for any
    $\bar{p}\in \H_h$, $\bar{p} = \hat{\bar{p}} + \bar{p}_{\perp}$, where $\hat{\bar{p}} \in \hat{\Z}_h$ and $\bar{p}_{\perp} \in \B_h$. Then $\Pi_h(\hat{\bar{p}}) = \bar{p}$, since $\hat{\bar{p}} = \bar{p} - \bar{p}_\perp$. Thus $\Pi_h$ is surjective.

    To show~\eqref{eq:dec-feec-harmonic} denote $\Pi_h(\hat{p}) - \hat{p}$ by $p_\perp$. By Lemma~\ref{lem:ip1}
    \[
    \norm{p_\perp}_F^2= |\langle \hat{p}, p_\perp\rangle_F - \langle \hat{p}, p_\perp \rangle_D| \le \hat{C}\,h\,\norm{\hat{p}}_V\;\norm{p_\perp}_V\, .
    \]
    Then we have $\norm{p_\perp}_F \le \hat{C}\,h \, \norm{\hat{p}}_F$, which proves the result.
\end{proof}

With Theorem~\ref{thm:harmonic-comparison} and DEC Poincar\'e inequality, we can prove the following stability theorem using DEC harmonic forms. We state the theorem without proof since the proof is exactly the same as the proof of \cite[Theorem 3.8]{ArFaWi2010} with FEEC Hodge decomposition replaced by DEC Hodge decomposition. 

\begin{theorem}[DEC stability using DEC norm and DEC harmonic forms]
\label{thm:inf-sup-DEC-norm}
    Let $\{X_h\}$ be a family of DEC-regular triangulations of $\Omega$. Then there exists a $\gamma > 0$, depending only on $\delta_0$ and $c_P$ in FEEC Poincar\'e inequality, such that for any $(\sigma, u, p) \in \whitney{k-1}{X_h} \times \whitney{k}{X_h} \times \hat{\H}_h^k$, when mesh size $h$ is small enough, there exists $(\tau, v, q) \in \whitney{k-1}{X_h} \times \whitney{k}{X_h} \times \hat{\H}_h^k$ with
    \[
    B_D(\sigma, u, p;\tau, v, q) \ge \gamma\, (\norm{\sigma}_{V_h}^D + \norm{u}_{V_h}^D + \norm{p})\,(\norm{\tau}_{V_h}^D+\norm{v}_{V_h}^D + \norm{q})\, .
    \]
\end{theorem}

We also have the convergence theorem with regards to DEC harmonic forms.
\begin{theorem}[DEC error using FEEC norm and DEC harmonic forms] \label{thm:error-DEC-H}
    Let $\{X_h\}$ be a family of DEC regular triangulations of $\Omega$. Let $(\hat{u}_h, \hat{\sigma}_h, \hat{p}_h)$ be the DEC solution and $(u_h, \sigma_h, p_h)$  the FEEC solution. Then we have the error estimate:
    \[
    \norm{\hat{\sigma}_h - \sigma_h}_V + \norm{\hat{u}_h - u_h}_V + \norm{\hat{p}_h - p_h}_F \le C\,h\,\big(\norm{\sigma_h}_V + \norm{u_h}_V + \norm{p_h}_F + \norm{f_h}_V\big)\, .
    \]
\end{theorem}

\begin{proof}
    First of all, we have
    \[
    B_D(\hat{\sigma}_h, \hat{u}_h, \hat{p}_h; \tau_h, v_h, \hat{q}_h) = \langle f_h, v_h\rangle_D, \quad  (\tau_h, v_h, \hat{q}_h) \in \whitney{k-1}{X_h} \times \whitney{k}{X_h} \times \hat{\H}_h^k
    \]
    and 
    \[
    B(\sigma_h, u_h, p_h; \tau_h, v_h, q_h) = \langle f_h, v_h\rangle_F, \quad  (\tau_h, v_h, q_h) \in \whitney{k-1}{X_h} \times \whitney{k}{X_h} \times \H_h^k
    \]
    Substituting all DEC inner product in first equation with FEEC inner product, we will have 
    \begin{equation*}
\begin{aligned}
    \big| B(\hat{\sigma}_h, \hat{u}_h, \hat{p}_h; \tau_h, v_h, \hat{q}_h) &- \langle f_h, v_h\rangle_F \big| \\
    &= \left| B(\hat{\sigma}_h, \hat{u}_h, \hat{p}_h; \tau_h, v_h, \hat{q}_h) - B_D(\hat{\sigma}_h, \hat{u}_h, \hat{p}_h; \tau_h, v_h, \hat{q}_h) 
    - \langle f_h, v_h\rangle_F + \langle f_h, v_h\rangle_D \right| \\
    &\le 
    \left| \langle\hat{\sigma}_h , \tau_h \rangle_D - \langle\hat{\sigma}_h , \tau_h \rangle_F \right|
    + \left| \langle \hat{u}_h , d\tau_h \rangle_D - \langle \hat{u}_h , d\tau_h \rangle_F \right| \\
    \quad &+ \left| \langle d\hat{\sigma}_h , v_h \rangle_D - \langle d\hat{\sigma}_h , v_h \rangle_F \right|
    + \left| \langle f_h, v_h\rangle_F - \langle f_h, v_h\rangle_D \right| \\
    \quad &+ \left| \langle d\hat{u}_h , dv_h \rangle_D - \langle d\hat{u}_h , dv_h \rangle_F \right|
    + \left| \langle \hat{p}_h , v_h \rangle_D - \langle \hat{p}_h , v_h \rangle_F \right| \\
    \quad &+ \left| \langle \hat{u}_h , \hat{q}_h \rangle_D - \langle \hat{u}_h , \hat{q}_h \rangle_F \right| \\
    &\le \hat{C}\,h \left( \norm{\hat{\sigma}_h}_V + \norm{\hat{u}_h}_V + \norm{\hat{p}_h}_F + \norm{f_h}_V \right)
    \left( \norm{\tau_h}_V + \norm{v_h}_V + \norm{\hat{q}_h}_V \right)\, .
\end{aligned}
\end{equation*}
    Then we have
    \begin{equation*}
        \begin{aligned}
            |B(\hat{\sigma}_h-\sigma_h, \hat{u}_h &- u_h, \Pi_h(\hat{p}_h)-p_h; \tau_h, v_h, \Pi_h(\hat{q}_h))| \le |B(\hat{\sigma}_h, \hat{u}_h, \hat{p}_h; \tau_h, v_h, \hat{q}_h) - \langle f_h, v_h\rangle_F | \\
            & + |\langle v, \Pi_h(\hat{p}_h) - p_h\rangle_F| + |\langle u, \hat{q}_h - \Pi_h(\hat{q}_h) \rangle_F| \\ 
            &\le 2\hat{C}\,h (\norm{\hat{\sigma}_h}_V + \norm{\hat{u}_h}_V + \norm{\hat{p}_h}_F + \norm{f_h}_V)(\norm{\tau_h}_V + \norm{v_h}_V + \norm{\hat{q}_h}_F) \, .
        \end{aligned}
    \end{equation*}
    By Theorem~\ref{thm:harmonic-comparison}, we have $\norm{\Pi_h(\hat{q}_h)}_V \ge (1-\hat{C}\,h) \norm{\hat{q}_h}_V $. When $h<1/(2\hat{C})$, we have $\norm{\Pi_h(\hat{q}_h)}_V \ge \frac{1}{2}\norm{\hat{q}_h}_V$. With this and \cite[Theorem 3.8]{ArFaWi2010}, we have 
    \begin{equation*}
        \begin{aligned}
            \norm{\hat{\sigma}_h - \sigma_h}_V 
            &+ \norm{\hat{u}_h - u_h}_V 
            + \norm{\hat{p}_h - p_h}_V 
            \le \norm{\hat{\sigma}_h - \sigma_h}_V 
            + \norm{\hat{u}_h - u_h}_V 
            + \norm{\Pi_h(\hat{p}_h) - p_h}_V 
            + \hat{C}\,h \norm{\hat{p}_h}_F \\
            &\le \frac{4\hat{C}\,h}{\gamma} 
            \left( \norm{\hat{\sigma}_h}_V 
            + \norm{\hat{u}_h}_V 
            + \norm{\hat{p}_h}_F 
            + \norm{f_h}_V \right) \\
            &\le \frac{4\hat{C}\,h}{\gamma} 
            \left( \norm{\hat{\sigma}_h - \sigma_h}_V 
            + \norm{\hat{u}_h - u_h}_V 
            + \norm{\hat{p}_h - p_h}_V 
            + \norm{\sigma_h}_V 
            + \norm{u_h}_V 
            + \norm{p_h}_F 
            + \norm{f_h}_V \right)\, .
        \end{aligned}
    \end{equation*}
    Then for $h \le \gamma/(8\hat{C})$, we have 
    \[
    \norm{\hat{\sigma}_h - \sigma_h}_V + \norm{\hat{u}_h - u_h}_V + \norm{\hat{p}_h - p_h}_F \le \frac{8\hat{C}\,h}{\gamma}(\norm{\sigma_h}_V + \norm{u_h}_V + \norm{p_h}_F + \norm{f_h}_V)\, .
    \]
\end{proof}

\begin{remark}
    Both using DEC harmonic forms or FEEC harmonic forms in the mixed formulation will give us the same convergence rate. 
\end{remark}



Since we know that for FEEC method, that $(\sigma_h, u_h, p_h)$ converges to the true solution, then we have our DEC solution also converges to the true solution.



\section*{Acknowledgement}
The work of ANH and CZ was supported in part by NSF DMS-2208581.
The work of KH was supported by a Royal Society University Research Fellowship (URF$\backslash$R1$\backslash$221398) and an ERC Starting Grant (project 101164551, GeoFEM).

\appendix
\section{Shape regularity}

\begin{lemma}
    Let $\{X_h\}$ be a family of meshes triangulating $\Omega$. Then $r/R$ is bounded below if and only if there is a  $\delta_0 > 0$ such that $\theta \ge \delta_0$ for all angles $\theta$ in the family. This is also equivalent to the ratios $\text{area}/R^2$ for triangles being bounded below.
\end{lemma}

\begin{proof}
    Let $T$ be some triangle in the family with angles $A, B, C$, opposite side lengths $a, b, c$ and area $\mu$. Then $r = \mu/(a+b+c)$, and $R = a/(2\sin A)$. Thus 
    \begin{equation*}
        \begin{aligned}
            \frac{r}{R} = \frac{2\mu \sin A}{a(a+b+c)} = \frac{ab\sin C \sin A}{a(a+b+c)} = \frac{\sin A^2\sin B \sin C}{\sin A (\sin A+\sin B +\sin C)} 
            = \frac{\sin A\sin B \sin C}{(\sin A+\sin B +\sin C)}
        \end{aligned}
    \end{equation*}
    If all angles are bounded by $\delta_0$, then $r/R \ge (\sin^3 \delta_0)/3$ showing that $r/R$ is bounded.

    To show the other direction assume that $r/R$ is bounded below by $\varepsilon > 0$. We need to show that $A$, $B$, and $C$ are bounded away from zero. Without loss of generality assume $A \leq B \leq C$. First, observe that
    \[
    \varepsilon \le \frac{r}{R} \leq \frac{\sin A \sin B \sin C}{3 \sin A}\, ,
    \]
    which implies $\sin B \geq \sin B \sin C \geq 3 \varepsilon$.
    Next, using the same inequality differently,
    \[
    \varepsilon \le \frac{r}{R} \leq \frac{\sin A \sin B \sin C}{\sin B} \leq \frac{\sin A \sin B \sin C}{3 \varepsilon},
    \]
    which gives $\sin A \geq 3 \varepsilon^2 > 0$.  Therefore, $A$ is bounded below by $\arcsin(3\varepsilon^2) > 0$. Since $A \leq B \leq C$, it follows that $B$ and $C$ are also bounded below by the same positive constant. Thus, all three angles are uniformly bounded away from zero. The proof of equivalence to $\text{area}/R^2$ being bounded below is similar and omitted here.
\end{proof}

\bibliography{convergence,decfem}
\bibliographystyle{plainnat}

\end{document}

%% file: figs/0-form-acute.tex
\begin{tikzpicture}[scale=4]

    \coordinate (v0) at (0.0, 0.0);
    \coordinate (v1) at (1.0, 0.0);
    \coordinate (v2) at (0.6, 0.8);
    \coordinate (O) at (0.5, 0.25);

    \coordinate (M_01) at ($ (v0)!0.5!(v1) $);
    \coordinate (M_02) at ($ (v0)!0.5!(v2) $);
  
  
    \draw[thin] (v0) -- (v1) -- (v2) -- cycle;
    \draw[thin] (O) -- (M_01);
    \draw[thin] (O) -- (M_02);
    \draw[dashed] (M_01) -- (M_02);
    \draw[dashed] (O) -- (v2);
    \draw[dashed] (O) -- (v1);
    \draw[dashed] (O) -- (v0);

    \node[below] at (v0) {$v_0$};
    \node[below] at (v1) {$v_1$};
    \node[above] at (v2) {$v_2$};
    \node[above right] at (O) {$O$};

    \node[below] at (M_01) {$M_1$};
    \node[above left] at (M_02) {$M_2$};

\end{tikzpicture}

%% file: figs/delaunay.tex
\begin{tikzpicture}[scale=4]

    \coordinate (v0) at (0.0, 0.0);
    \coordinate (v1) at (1.0, 0.0);
    \coordinate (v2) at (0.6, 0.8);
    \coordinate (v3) at (0.5, -0.4);
  
    \coordinate (M_01) at ($ (v0)!0.5!(v1) $);
    \coordinate (M_02) at ($ (v0)!0.5!(v2) $);
    \coordinate (M_12) at ($ (v1)!0.5!(v2) $);
    \coordinate (M_03) at ($ (v0)!0.5!(v3) $);
    \coordinate (M_13) at ($ (v1)!0.5!(v3) $);
  
  
    \draw[thin] (v0) -- (v1) -- (v2) -- cycle;
    \draw[thin] (v1) -- (v3) -- (v0);
  
    \node[left] at (v0) {$v_0$};
    \node[right] at (v1) {$v_1$};
    \node[above] at (v2) {$v_2$};
    \node[below] at (v3) {$v_3$};
    \node[above right, xshift=5pt] at (v0) {$A_1$};
    \node[below, xshift=-1pt, yshift=-5pt] at (v2) {$C_1$};
    \node[above left, xshift=-3pt] at (v1) {$B_1$};
    \node[above, xshift=-1pt, yshift=2pt] at (v3) {$C_2$};
    \node[above] at (M_01) {$c$};
    \node[above left] at (M_02) {$b_1$};
    \node[above right] at (M_12) {$a_1$};
    \node[below left] at (M_03) {$b_2$};
    \node[below right] at (M_13) {$a_2$};

\end{tikzpicture}

%% file: figs/0-form-Delaunay-internal.tex
\begin{tikzpicture}[scale=4]

    \coordinate (v0) at (0.0, 0.0);
    \coordinate (v1) at (1.0, 0.0);
    \coordinate (v2) at (0.6, 0.8);
    \coordinate (O) at (0.5, 0.25);

    \coordinate (M_01) at ($ (v0)!0.6!(v1) $);
    \coordinate (M_02) at ($ (v0)!0.5!(v2) $);
  
  
    \draw[thin] (v0) -- (v1) -- (v2) -- cycle;
    \draw[thin] (O) -- (M_01);
    \draw[thin] (O) -- (M_02);

    \draw[thin] (O) -- (v2);
    \draw[thin] (O) -- (v1);
    \draw[thin] (O) -- (v0);

    \node[below] at (v0) {$v_3$};
    \node[above right, xshift=13pt, yshift=-2.5pt] at (v0) {$C_2$};
    \node[above right, xshift=13pt, yshift=11pt] at (v0) {$B_3$};
    \node[below] at (v1) {$v_1$};
    \node[above left, xshift=-2pt, yshift=4pt] at (v1) {$C_5$};
    \node[above left, xshift=-10pt, yshift=-3pt] at (v1) {$B_1$};
    \node[above] at (v2) {$v_5$};
    \node[below left, xshift=-1pt, yshift=-16pt] at (v2) {$C_4$};
    \node[below right, xshift=-4pt, yshift=-16pt] at (v2) {$B_5$};
    \node[above right] at (O) {$v_0$};

    \node[below] at (M_01) {$v_2$};
    \node[above right, xshift=-2pt, yshift=-2pt] at (M_01) {$C_1$};
    \node[above left, xshift=-1pt, yshift=-2pt] at (M_01) {$B_2$};
    \node[above left] at (M_02) {$v_4$};
    \node[below right, xshift=-6pt, yshift=-2pt] at (M_02) {$C_3$};
    \node[below right, xshift=0pt, yshift=9pt] at (M_02) {$B_4$};

\end{tikzpicture}

%% file: figs/0-form-Delaunay-boundary.tex
\begin{tikzpicture}[scale=4]

    \coordinate (v0) at (0.0, 0.0);
    \coordinate (v1) at (-0.1, 0.4);
    \coordinate (v3) at (0.8, 0.0);
    \coordinate (v2) at (0.48, 0.64);
    \coordinate (v4) at (0.4, -0.32);
  
    \draw[thin] (v0) -- (v3) -- (v2) -- cycle;
    \draw[thin] (v3) -- (v4) -- (v0);
    \draw[thin] (v0) -- (v1) -- (v2);

    \node[below] at (v0) {$v_0$};
    \node[below, xshift=2pt] at (v3) {$v_3$};
    \node[above] at (v2) {$v_2$};
    \node[below] at (v4) {$v_4$};
    \node[left] at (v1) {$v_1$};
    \node[right, yshift=-3pt] at (v1) {$B_1$};
    
    \node[below, xshift=-2pt, yshift=-5pt] at (v2) {$B_2$};
     \node[below left, xshift=-9pt, yshift=-6pt] at (v2) {$C_1$};
    \node[above left, xshift=-3pt] at (v3) {$C_2$};
    \node[below left, yshift=2pt, xshift=-8pt] at (v3) {$B_3$};
    \node[above, xshift=-2pt, yshift=2pt] at (v4) {$C_3$};

\end{tikzpicture}